\documentclass[11pt]{amsart}
\usepackage[margin=1.2in, top=1.2in, bottom=1.2in]{geometry}
\usepackage{hyperref}
\usepackage{amsmath}
\usepackage{amsfonts}
\usepackage{amssymb}
\usepackage{graphicx}
\usepackage{mathrsfs}
\usepackage{graphicx,color}
\usepackage{amsmath,amsfonts,amsthm,enumerate,amscd,latexsym,curves}
\usepackage[mathscr]{eucal}
\usepackage{caption}
\usepackage{subcaption}
\usepackage{color}
\usepackage{comment}
\usepackage{tikz-cd}
\usepackage{graphicx,color}
\usepackage[all]{xy}
\usepackage{mathrsfs}
\usepackage{marvosym}
\usepackage{mathtools,tikz-cd}
\usepackage{graphicx}
\usepackage{enumerate}
\usepackage{amsmath}

\usepackage{amsmath}    % for \underset
\usepackage{amssymb}    % for math symbols
\usepackage{tikz}       % for better circled numbers

\usepackage{tikz-cd}
\usepackage{tikz}
\usetikzlibrary{shapes.geometric, arrows}
\tikzstyle{startstop} = [rectangle, rounded corners, minimum width=2.5cm, minimum height=0.5cm,text centered, text width=2cm, draw=black, fill=white!30]
\tikzstyle{startstop2} = [rectangle, rounded corners, minimum width=2.5cm, minimum height=0.5cm,text centered, text width=2cm, draw=black, fill=white!30]
\tikzstyle{startstop3} = [rectangle, rounded corners, minimum width=2.5cm, minimum height=0.5cm,text centered, text width=2.5cm, draw=black, fill=white!30]
\tikzstyle{arrow} = [thick,->,>=stealth]

\newtheorem{theorem}{Theorem}[section]
\newtheorem{lemma}[theorem]{Lemma}
\newtheorem{question}[theorem]{Question}
\newtheorem{corollary}[theorem]{Corollary}
\newtheorem{proposition}[theorem]{Proposition}
\newtheorem{definition}[theorem]{Definition}
\newtheorem{example}[theorem]{Example}
\newtheorem{remark}[theorem]{Remark}

\newtheorem{construction}[theorem]{Construction}

\numberwithin{equation}{section}
\numberwithin{figure}{section}

\newcommand{\ZZ}{\mathbb{Z}}

\newcommand{\RR}{\mathbb{R}}

\newcommand{\Z}{\mathbb{Z}}
\newcommand{\Q}{\mathbb{Q}}

\newcommand{\PP}{\mathbb{P}}

\newcommand{\CC}{\mathbb{C}}
\newcommand{\CP}{\mathbb{CP}}

\newcommand{\cM}{\mathcal{M}}
\newcommand{\cO}{\mathcal{O}}

\newcommand{\cE}{\mathcal{E}}

\newcommand{\cJ}{\mathcal{J}}

\newcommand{\cC}{\mathcal{C}}

\newcommand{\Stab}{\mathrm{Stab}}

\newcommand{\Ham}{\mathrm{Ham}}
\newcommand{\Symp}{\mathrm{Symp}}
\newcommand{\Diff}{\mathrm{Diff}}

\newcommand{\im}{\mathrm{im}}

\newcommand{\reg}{\mathrm{reg}}

\newcommand{\JJ}{\mathcal{J}}

\newcommand{\mJ}{\mathcal{J}}

\newcommand{\mC}{\mathcal{C}}

\newcommand{\mG}{\mathcal{G}}

\newcommand{\aA}{\mathbb{A}}

\newcommand{\DD}{\mathbb{D}}
\newcommand{\EE}{\mathbb{E}}
\newcommand{\RP}{\mathbb{RP}}
\newcommand{\id}{\mbox{id}}

\newcommand{\w}{\omega}

\newcommand{\ov}{\overline}

\newcommand{\wt}{\widetilde}

\begin{document}

\title
[$C^0$-rigidity of Hamiltonian diffeomorphism groups]
%{The $C^0$ closure of $\Ham$ in $\Symp$ of rational symplectic 4-manifolds}
{$C^0$-rigidity of the Hamiltonian diffeomorphism group of symplectic rational surfaces}

\author{Marcelo Atallah, Cheuk Yu Mak and Weiwei Wu}

\maketitle

\begin{abstract} 

We investigate the $C^0$-topology of the group of symplectic diffeomorphisms of positive symplectic rational surfaces. For all but a few exceptions, we prove that the group of Hamiltonian diffeomorphisms forms a connected component in the $C^0$-topology. This provides the first nontrivial case in which the group of Hamiltonian diffeomorphisms is known to be $C^0$-closed inside the group of symplectic diffeomorphisms.

The key to our approach is to build a bridge between techniques from symplectic mapping class groups and problems in $C^0$-symplectic topology. Via a careful adaptation of tools from $J$-holomorphic foliation and inflation, we establish the necessary $C^0$-distance estimates.   We hope that this serves as an example of how these two subfields can interact fruitfully, and also propose several questions arising from this interplay.

% We study the $C^0$ topology of the group $\Symp(X,\omega)$ of symplectic diffeomorphisms of a positive symplectic rational surface. When $(X,\omega)$ is not a small blow-up of the monotone $\CC P^2\#k\overline{\CC P}^2$, $k=6,7,8$, we show that the group $\Ham(X,\omega)$ of Hamiltonian diffeomorphisms is a connected component of $\Symp(X,\omega)$ in the $C^0$ topology. This is the first non-trivial instance of a positive answer to the question of whether $\Ham(X,\omega)$ is $C^0$-closed in $\Symp(X,\omega)$.\\
% \textcolor{red}{[Is still don't love the abstract. Main issues: (1) Don't like the sentence ``When $(X,\omega)$ is not a small blow-up..."; (2) too many math symbols. Here is a suggestion:]}\\

\end{abstract}

\section{Introduction}
Let $(M,\omega)$ be a closed symplectic manifold and $d$ the distance induced by a fixed choice of Riemannian metric on $M$. The group $\Diff(M)$ of smooth automorphisms of $M$ is naturally equipped with the metric topology corresponding to the $C^0$-distance:\[d_{C^0}(\phi,\psi)=\max_{x\in M}d(\phi(x),\psi(x)),\] commonly referred to as the \textit{$C^0$-topology}. 
There is a natural nested inclusion of subgroups
\[
\Ham(M) \underset{(2)}{\subset} \Symp_0(M,\omega) \underset{(3)}{\subset} \Symp(M,\omega) \underset{(1)}{\subset} \Diff(M),
\]
with the induced subspace topology. The origin of $C^0$-symplectic topology can be traced back to the study of inclusion $(1)$, whereby Gromov-Eliashberg established their foundational theorem asserting that the group of symplectic diffeomorphisms $\Symp(M,\omega)$ is $C^0$-closed in $\Diff(M)$.
  This discovery revealed a striking $C^0$-rigidity phenomenon giving rise to the field that has since seen significant development; see \cite{RSV21, C-GHS, HLS16, HLS15, Sh22-2, Sh22-1, MO07, Op09, BO16, BHS18, BHS21, LMP1,Jannaud,SSmemoire} for a few examples. 
  
Inclusion $(2)$ marks a central open problem in $C^0$ symplectic topology, called the $C^0$-flux conjecture.  The conjecture predicts that the Hamiltonian diffeomorphism group $\Ham(M,\omega)$ is also $C^0$-closed in the identity component $\Symp_0(M,\omega)$ of $\Symp(M,\omega)$.  While the conjecture remains unresolved in full generality, substantial progress has been made in the works of Lalonde-McDuff-Polterovich \cite{LMP1}, Buhovsky \cite{LBflux}, including recent developments in the announced ongoing work of the first-named author and Shelukhin \cite{ME-WIP}; see also \cite{Ono_FluxConjecture} for the $C^1$-flux conjecture and \cite{ACLS} for discussion about the Lagrangian analogue of this question.  

Notably, discussions of the $C^0$-flux conjecture are often accompanied by a caveat pertaining to inclusion $(3)$, which is the main question we will investigate in this paper:

\begin{question}
\label{quest:C0-closure}
Is ${\Symp}_0(M,\omega)$ closed in $\Symp(M,\omega)$ in the $C^0$-topology?
\end{question}

Note that the $C^1$ counterpart of Question \ref{quest:C0-closure} is trivial, so it is a problem exclusive to $C^0$ topology.  This problem is closely related to the continuity of the flux homomorphism.  Recall  the  flux homomorphism $\textbf{Flux}(\{\psi_t\}):=\int_0^1[\iota(X_t)\w]dt\in H^1(M,\RR)$.  If \textbf{Flux} is $C^0$-continuous, some instances of which are proven in \cite{ME-WIP}, and assuming Question \ref{quest:C0-closure} has a negative answer for $M$, it implies \textbf{Flux} extends naturally to a strictly larger subgroup $\ov\Symp_0(M,\w)\subset \Symp_h(M,\w)$ which contains $\Symp_0(M,\w)$.  Note that the continuity of the flux homomorphism implies the $C^0$-flux conjecture.

Question \ref{quest:C0-closure} already 
appeared in \cite[Section 1]{LMP1} and \cite[Section 10.2.18]{MSIntro} when the $C^0$-flux conjecture was posed.  It is straightforward to see that the $C^0$-closure of $\Symp_0(M,\omega)$ in $\Symp(M,\omega)$ is contained in the \textit{symplectic Torelli group} $\Symp_h(M,\omega)$, which consists of symplectic diffeomorphisms that act trivially on homology. Indeed, this follows from the local-path-connectedness of $\mathrm{Homeo}(M)$ in the $C^0$-topology \cite{EKHomeo}. Therefore, an affirmative answer to Question \ref{quest:C0-closure} follows whenever the \textbf{symplectic mapping class group} $\pi_0(\Symp_h(M,\omega))$ is trivial. This is known to hold for a few cases such as: $\CP^2$ with the Fubini-Study form, the $S^2\times S^2$ and its $k$-point blow-ups, $1\leq k\leq3$.  More generally, any \textit{positive symplectic rational surface} with $\omega$ of type $\aA$ (as in Definition \ref{d:typeD}) has trivial symplectic mapping class groups. There is a very long list of works in this direction, see \cite{Gromov85,LP04,Sei08,Ev11,AP2013, LLWsmall, LLW22,BL2023}, as a small portion of closely related works. Moreover, when $M$ is a closed orientable surface, a positive answer to Question \ref{quest:C0-closure} follows from the fact that the inclusion of the group of area-preserving diffeomorphisms into the group of orientation-preserving homeomorphisms induces an isomorphism between the symplectic mapping class group and the standard mapping class group. To our knowledge, no results have been obtained for other symplectic manifolds in the literature, especially those with non-trivial symplectic mapping class group $\pi_0(\Symp_h(M,\omega))$. 

Recall that a \textit{positive symplectic rational surface} $(X,\omega)$ is either a symplectic $S^2\times S^2$ or $\CP^2\#n\overline\CP^2$, for $n\geq 0$, satisfying $c_1\cdot[\omega]>0$. In \cite{LLW22}, the symplectic classes of positive rational symplecitc surfaces are classified into types $\aA$, $\DD$ and $\EE$ (see Definition \ref{d:typeD}).
 Furthermore, it is shown that when $\omega$ is of type $\DD$, then $\pi_0(\Symp_h(X,\omega))$ is the quotient of the pure braid group of $k$ strands\footnote{The integer $k$ depends on the subtype $\DD_{k}$ of the type $\DD$ symplectic classes.} on the sphere by $\ZZ/2\mathbb{Z}$, so in particular, Question \ref{quest:C0-closure} has non-trivial content.
 The main result of this paper is the following.
\begin{theorem}
\label{thm:main}
Let $(X,\omega)$ be a positive symplectic rational surface of type $\DD$. Then, the group of Hamiltonian diffeomorphisms $\Ham(X,\omega)$ is a connected component of $\Symp(X,\omega)$ in the $C^0$-topology; in particular, it is closed.    
\end{theorem}

This provides a positive answer to Question \ref{quest:C0-closure}, as $(X,\omega)$ is simply-connected, implying that $\Symp_0(X,\omega)$ coincides with $\Ham(X,\omega)$. 
We remark that most symplectic forms on a positive symplectic rational surface are of type $\aA$ or $\DD$; see Remark \ref{r:EE}.

\subsection*{Relations to other works}

Our proof borrows ideas from the study of the topology of symplectomorphism groups, which is initiated by Gromov \cite{Gromov85}.  This direction of research has seen significant progress by many different authors, and here we only give a highly incomplete list of the subsequent studies relevant to the current paper \cite{Abr98,AM00,Anj02,LP04,Pin08,Buse11,Hind12}.

The key observation of our proof is to use a divisorial decomposition of the ambient symplectic manifold $M$, so that we could acquire a $C^0$ control of isotopies of certain symplectic divisors (\ref{ssec:StepAB}, Step A).  In light of this, we propose Question \ref{q:C0smallIsotopy} in section \ref{sec:C0isotopy}, which might be of independent interest.  In other words, we ask whether the space of symplectic embedding of a divisor is locally path-connected.

This in turn is related to another long-standing open problem concerning the local path-connectedness of the Hamiltonian group \cite{Ba78}.  To date, the group of compactly supported Hamiltonian diffeomorphisms $\Ham_c(M,\omega)$ is known to be locally path-connected in the $C^0$ topology if $(M,\omega)$ is a closed surface or the standard symplectic $2n$-Ball; see \cite{Fathi, Sey13}. To our knowledge, local path-connectedness remains wide open outside these cases.  Provided the answer to Question \ref{q:C0smallIsotopy} is positive, we should be able to use the divisorial decomposition of $M$ to deduce a locally path-connectedness result of $\Symp(M)$.  This will be the topic of study in the sequel.  Further connections of this problem to Floer theory is discussed in Section \ref{sec:Floer}.

%Moreover, we believe that the methods used to prove Theorem \ref{thm:main} could potentially establish, under the same assumptions, that $\Symp(X,\omega)$ is locally path-connected in the $C^0$-topology. 

\subsection*{Roadmap of the proof}
We prove Theorem \ref{thm:main} by showing that there exists $\epsilon>0$ such that if $f\in\Symp(X,\omega)$ satisfies $d_{C^0}(f,\id)<\epsilon$ then $f$ is a Hamiltonian diffeomorphism.

Historically, the symplectic mapping class group of rational surfaces has been studied by understanding the action of $\Symp(X,\omega)$ on the space of a certain type of symplectic divisor $\Sigma\subset X$ for which the group $\Symp_c(X\setminus\Sigma,\omega)$ of compactly supported symplectomorphisms of the complement $X\setminus\Sigma$ is weakly-contractible; see \cite{Gromov85} and \cite{Ev11}. We note that weakly-contractibility of $\Symp_c(X\setminus\Sigma,\omega)$ is a strong condition. More recently, techniques have been developed to partially drop this assumption. We now have a good understanding in the setting of positive rational surfaces; see \cite{LLW22} and \cite{LLWsmall}.

From these works, one can deduce a criterion, Theorem \ref{t:stab0ham}, to establish whether a symplectic diffeomorphism is Hamiltonian for rational surfaces of interest. It is shown that, for certain $\omega$ (as in Proposition \ref{p:stab0ham}), if $f$ fixes $\Sigma$ point-wise, then it must be Hamiltonian. In Section \ref{sec:Stab0} we generalize this criterion to all symplectic forms of pure type $\DD$, without knowing $\Symp_c(X\setminus\Sigma,\omega)$, using a family inflation technique as in \cite{McDuffErratum, LU06, Buse11, MO15} in addition to the bihopfian property of the mapping class group of $S^2$ with finitely many marked points.

In order to utilize the criterion, the challenge lies in constructing a Hamiltonian diffeomorphism $\phi$ such that: $(1)$ $\phi\circ f(\Sigma)=\Sigma$ and $(2)$ $\phi\circ f|_{\Sigma}$ is in the trivial mapping class group of $\Sigma$. We note it is possible to achieve $(1)$ by means of a standard $J$-isotopy argument. However, such an isotopy lacks the $C^0$ control to attain $(2)$ because $f_*J$ is not $C^0$ close to $J$ in the space of almost complex structures. To overcome this difficulty, we observe that, despite $f_*J$ not being controlled, the induced fibration structures of $J$ and $f_*J$ are indeed $C^0$ close.  This allows us to carefully choose a $J$-isotopy which has enough $C^0$ control over a prescribed region compatible with (but not containing) $\Sigma$. Analyzing the induced ambient isotopy through the projection induced by the fibration allows us to achieve $(2)$.    

We remark that there is a possibility of using Floer-theory as an alternative criterion. In Section \ref{sec:Floer}, we explain how to combine results in the literature to get the analogue of Theorem \ref{thm:main} for $A_n$-Milnor fibres; cf. \cite{Jannaud} for related results. Further discussion is included in Section \ref{sec:further}.

%\noindent \emph{\textbf{Acknowledgments.}}
    
\subsection*{Acknowledgments}
We thank Richard Hind, Vincent Humili\`ere, Tian-Jun Li, Jun Li, Sobhan Seyfaddini, and Egor Shelukhin for their interest and encouragement in this project.

M.A and C.Y.M are partially supported by the Royal Society University Research Fellowship. W.W. is supported by NSFC Grant 12471063.

\section{Symplectic Mapping class groups of positive rational surfaces}

In this section, we recall the main definitions and results of \cite{LLW22}, especially for pure type $\DD$ forms.

\subsection{Preliminaries}
Consider a rational symplectic surface $X_n\cong (\CP^2\#n\ov\CP^2,\w)$ which is a symplectic $4$-manifold obtained by a sequence of symplectic blow-ups from a standard $\CP^2$.  Let $\cE_\w$ denote the set of $\w$-exceptional classes, that is, those classes with self-intersection $(-1)$ which are represented by embedded $\w$-symplectic spheres.  A set of basis elements of its second homology classes is given by a line class $H$ and $n$ orthogonal exceptional classes $E_1,\cdots, E_n\in \cE_\w$. This basis is called a \textbf{canonical basis} if $$\w(E_i)=\min\{\w(E)\,\,|\,E\in \cE_\w,\,\,E\cdot E_k=0, k\ge i+1\}.$$ 

In particular, $E_n$ is an exceptional class with the minimal $\w$-area.  The existence of a canonical basis is guaranteed by Gromov compactness; see \cite{LW12}.  By abuse of notation, we will not distinguish between a cohomology class of degree $2$ and its Poincaré dual below.  Therefore, by fixing a canonical basis once and for all, the cohomology class $[\w]$ of $X_n$, up to rescaling, takes the form of 

\[
  [\w]=H-\sum_{i=1}^n a_iE_i.
\]

We denote this class by $[\w]=(1|a_1,\cdots,a_n)$ for $a_i\ge a_{i+1}$ and the anti-canonical class $c_1(\w)=(3|1,\cdots,1)$. In this paper, we also require the following positivity condition:

\begin{equation}\label{e:positivity}
    c_1\cdot[\w]>0.
\end{equation}   

A symplectic classes $[\w]$ that satisfies \eqref{e:positivity} is called a \textit{positive} symplectic class.  This condition is equivalent to the condition that $(X,\w)$ admits a smooth symplectic divisor $D$, such that $(X,D)$ forms a log Calabi-Yau pair; see \cite[Proposition 2.2]{LLW22}. 
Following \cite{LLW22} (see the discussion between Definition 2.13 and Lemma 2.14 of \cite{LLW22}), we use the following definition.

% There is a convenient way to characterize the type of a symplectic form on a canonical basis as follows. 

% It yields the following equivalent description for a symplectic form of type $\mathbb{D}$, which is of specific interest to us. 

\begin{definition}\label{d:typeD}
    We say a symplectic form $\w$ is of \textbf{type} $\mathbb{D}$, if under a canonical basis,
    \begin{align}\label{e:typeD}
         [\w]&=(1|\lambda,\underbrace{\frac{1-\lambda}{2},\cdots,\frac{1-\lambda}{2}}_{k},a_{k+2},\cdots,a_n),\quad a_{k+2}<\frac{1-\lambda}{2},\quad \text{ and } \\
         & \text{ either } \quad \lambda\in({1}/{3}, 1), k \ge 4, \quad  \text{ or }\quad \lambda={1}/{3}, k=4. \nonumber
    \end{align}
When $\lambda\in({1}/{3}, 1)$, we call it type $\DD_k$ and when $\lambda={1}/{3}$, we call it type $\DD_5$. If $n=k+1$, we say $[\w]$ is of \textbf{pure type} $\mathbb{D}$. A form is of type $\mathbb{E}$ if $\lambda={1}/{3}$ and $k=5,6,7$, and type $\mathbb{A}$ if it is neither type $\mathbb{D}$ nor $\mathbb{E}$. 
\end{definition}

\begin{example}
The manifold $\CP^2\#n\ov\CP^2$ with a monotone symplectic form $\w$ satisfies (after normalization) 
\[
[\w]=(1|\underbrace{\frac{1}{3},\cdots,\frac{1}{3}}_{n}).
\]
Therefore, it is of pure type $\aA$ if $n \le 4$, pure type $\DD$ if $n=5$ and pure type $\EE$ if $n=6,7,8$\footnote{Note that, when $\lambda=\frac{1}{3}$, $\frac{1-\lambda}{2}=\frac{1}{3}$ as well so there are $k+1$ many $\frac{1}{3}$ in total in \eqref{e:typeD}}. 
There is no monotone symplectic form when $n>8$. Note that the non-monotone symplectic form $$[\w]=(1|\lambda,\underbrace{\frac{1-\lambda}{2},\cdots,\frac{1-\lambda}{2}}_{4})$$ on $\CP^2\#5\ov\CP^2$ is also of pure type $\DD$. Note that it is of type $\DD_4$ as opposed to $\DD_5$ as in the monotone case.
\end{example}

\begin{remark}\label{r:EE}
    Most of the symplectic forms in $\CP^2\#n\ov\CP^2$ are of type $\aA$, and most of the remaining forms are of type $\DD$.
    Type $\EE$ symplectic forms only arise after monotone $6,7$ or $8$ point blow-ups and then (possibly) perform further small blow-ups. 
\end{remark}

\begin{remark}
    The \textit{type} of a symplectic form is originally defined by an associated Dynkin diagram of the Lagrangian spherical classes.  See \cite{LLW22} for more details.  More precisely, given any symplectic form $\w$, under a canonical basis, we consider the root system generated by classes $$\{H-E_1-E_2-E_3, E_i-E_{i+1}\}_{i=1}^{n-1}.$$  This root system is not simple, but the sub-root system (called a \textbf{Lagrangian root system}) generated by classes with the additional requirement $\w(L)=0$ is always a direct sum of root systems of type $\mathbb{A}, \mathbb{D}$ or $\mathbb{E}$.  We say that a symplectic form is of \textbf{type $\mathbb{D}$ or $\mathbb{E}$} if such a root system contains a direct summand of type $\mathbb{D}$ or $\mathbb{E}$, respectively; and we call it \textbf{type $\mathbb{A}$} if all direct summands are of type $\mathbb{A}$. Root systems obtained this way cannot have one summand of type $\mathbb{D}$ and another summand of type $\mathbb{E}$, so the type is well-defined.

% $[\w]=(1|a_1,\cdots,a_n)$ satisfies the following reducedness condition.

% \begin{equation}\label{e:reduced}
%   \begin{cases}
%     1\ge c_1+c_2+c_3, \\
%     c_i\ge c_{i+1}>0.
%   \end{cases}
% \end{equation}

% in \cite{LLW22}.  Explicitly, we are interested in the case of type $D$, which has a simple numerical description.

\end{remark}

$J$-holomorphic curves on rational surfaces have some significant properties. Denoting the space of $\w$-compatible almost complex structures by $\cJ_\w$, the following two facts will be useful for our upcoming proofs.

\begin{lemma}[{\cite[Lemma 3.9]{LLW22}; \cite[Proposition 6]{LZ15}; \cite[Proposition 3.5]{Zha17}}]\label{l:nefemb}
    % Given an almost complex structure $J\in\cJ_\w$, we say a homology class $A$ is $J$-nef if $A\cdot [\Sigma]\ge0$ for any irreducible $J$-curve $\Sigma$. 

    % Then $A$ has an embedded representative.  In particular, f

    Let $X$ be a rational surface with $\chi(X)>4$ and symplectic class $[\w]=(1|a_1,\cdots,a_n)$ under a canonical basis.\footnote{Also termed a reduced symplectic form in the terminology of \cite{LLW22}.} Then, $H-E_1$ admits an embedded $J$-holomorphic representative for any $J\in\cJ_\w$.  In particular, one has a $J$-holomorphic foliation by curves in the class $H-E_1$.
\end{lemma}

% \textcolor{red}{[$\cJ_\w$ has not been defined yet. Also might be good to say what irreducible means in this context. I would suggest moving the definitions outside of the statement of the Lemma.] [Make the statement minimal]}
% \medskip

% \textcolor{red}{[Another possibility is to state the Lemma in a more concise form. Also, a comment (outside the statement of the Lemma) on the Euler characteristic would be helpful.]}

\begin{lemma}[\cite{Pin08}, Lemma 1.2]\label{l:minimal}
    Let $X$ be a symplectic rational surface, and let $E$ be an exceptional class with the smallest area among all exceptional classes.  Then $E$ admits an embedded $J$-holomorphic representative for any almost complex structure $J\in \cJ_\w$.
\end{lemma}

As a consequence of Lemma \ref{l:minimal}, if $[\w]=(1|a_1,\cdots,a_n)$ is a symplectic class under a canonical basis, then $E_n$ has an embedded pseudo-holomorphic representative for all $J\in\cJ_\w$.

% \textcolor{red}{[Might be good to add a few words about how the above statement follows from Lemma \ref{l:nefemb}]}

\subsection{Main SMCG results}

The main theorems of \cite{Ev11}, \cite{LLWsmall}, and \cite{LLW22} can be summarized as follows. Let $PB_k(S^2)$ be the pure braid group of $k$ strands on $S^2$, and $\Symp_h(X,\w)$ be the symplectic Torelli group (i.e. the subgroup of the symplectomorphism group consisting of elements which act  trivially on homology).

\begin{theorem}\label{theorem:SMCG}
   If $[\w]$ is of type $\DD_k$, then $$\pi_0(\Symp_h(X,\w))\cong PB_{k}(S^2)/\mathbb{Z}_2,$$ is generated by Dehn twists of Lagrangian spheres.
\end{theorem}

The proof of this theorem, which we briefly recall below, is vital to the proof of our main result.  We refer interested readers to \cite{Ev11} for the monotone case of $n=5$, to \cite{LLWsmall} for the non-monotone case for $n=5$, and \cite{LLW22} for the case when $n\ge6$.   In the rest of this section, we consider exclusively pure type $\DD$ forms, and postpone non-pure forms to later sections.  

% \subsection*{Pure Type $\mathbb{D}$ forms}

Consider a symplectic configuration inside $X$, which is a union of smooth symplectic divisors that is $J$-holomorphic for some $J\in\cJ_\w$.  

We consider configurations consisting of exceptional symplectic spheres in $X= \CC P^2  \# n\overline {\CC P^2}$ with transversal and positive pairwise intersections with homology classes shown in Figure \ref{Cn} when $n\ge6$, and Figure \ref{conf5} when $n=5$.

\begin{figure}[ht]
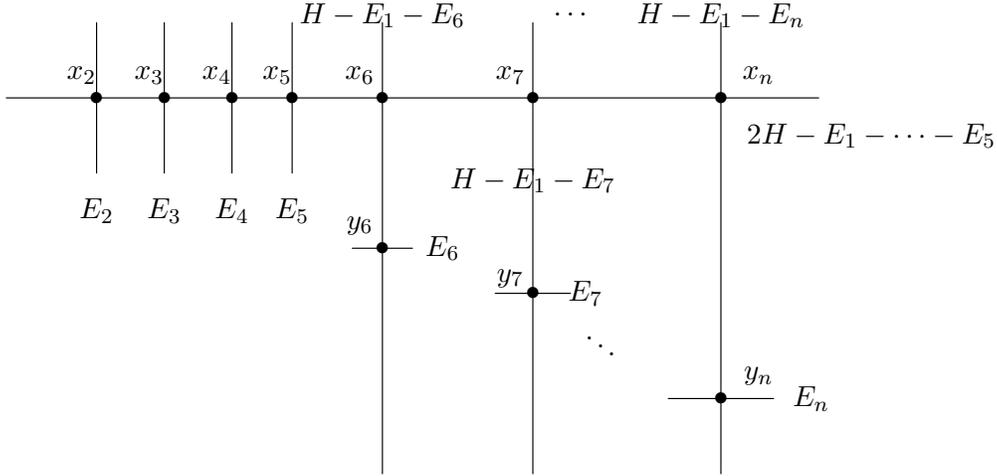

  \centering
\[
\xy
(0, -10)*{};(108, -10)* {}**\dir{-};
(115, -10)*{};
(115,-15)*{2H-E_1-\cdots-E_5};
(50, 0)*{}; (50, -60)*{}**\dir{-};
(50, 1)*{H-E_1-E_6};
(12, 0)*{}; (12, -20)*{}**\dir{-};
(10, -7)*{x_2};
(12, -10)*{\bullet};
(12, -25)*{E_2};
(21, 0)*{}; (21, -20)*{}**\dir{-};
(19, -7)*{x_3};
(21, -10)*{\bullet};
(21, -25)*{E_3};
(30, 0)*{}; (30, -20)*{}**\dir{-};
(28, -7)*{x_4};
(30, -10)*{\bullet};
(30, -25)*{E_4};
(38, 0)*{}; (38, -20)*{}**\dir[red, ultra thick, domain=0:6]{-};
(36, -7)*{x_5};
(38, -10)*{\bullet};
(38, -25)*{E_5};
(46, -30)*{}; (54, -30)*{}**\dir{-};
(47, -7)*{x_6};
(50, -10)*{\bullet};
(50, -30)*{\bullet};
(47,-27)*{y_6};
(58, -30)*{E_6};
(70, 0)*{}; (70, -60)*{}**\dir{-};
(70, -21)*{H-E_1-E_7};(75, 1)*{\cdots};
(95, 0)*{}; (95, -60)*{}**\dir{-};
(95, 1)*{H-E_1-E_n};
(65, -36)*{}; (75, -36)*{}**\dir{-};
(67, -7)*{x_7};
(70, -10)*{\bullet};
(70, -36)*{\bullet};
(95, -10)*{\bullet};
(95, -50)*{\bullet};
(67,-34)*{y_7};
(77, -36)*{E_7};
(88, -50)*{}; (102, -50)*{}**\dir{-};
(107, -50)*{E_n};
(100, -7)*{x_n};
(100, -47)*{y_n};
(79, -42)*{\ddots};
\endxy
\]
\caption{A filling divisor $\Sigma$ in $X$, $n\geq 6$}
  \label{Cn}
\end{figure}

\begin{figure}[ht]
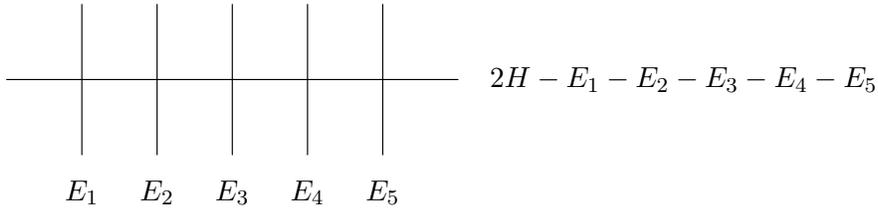

\begin{center}
\[
\xy
(0, -10)*{};(60, -10)* {}**\dir{-};
(90, -10)* {2H-E_1-E_2-E_3-E_4-E_5};
(10, 0)*{}; (10, -20)*{}**\dir{-};
(10, -25)*{E_1};
(20, 0)*{}; (20, -20)*{}**\dir{-};
(20, -25)*{E_2};
(30, 0)*{}; (30, -20)*{}**\dir{-};
(30, -25)*{E_3};
(40, 0)*{}; (40, -20)*{}**\dir{-};
(40, -25)*{E_4};
(50, 0)*{}; (50, -20)*{}**\dir[red, ultra thick, domain=0:6]{-};
(50, -25)*{E_5};
\endxy
\]
  \caption{A filling divisor for $n=5$ monotone case}\label{conf5}
\end{center}
\end{figure}

Such a symplectic configuration is called a \textbf{filling divisor}.  Under additional assumptions, we will see that its complement is a Weinstein manifold, whose compactly supported symplectomorphism group is weakly contractible.  The set of filling divisors is denoted by $\mC_\w$.

% \textcolor{red}{[There was a bit of a jump here. It would be good to define configuration and maybe motivate the term ``filling divisor".]}

% Occasionally, we will use $\mC_n$ to denote the set of homology classes of a filling divisor when the context is clear. 

Take a subset $\mC^0_\w \subset \mC_\w$, which consists of configurations whose components have pairwise symplectically orthogonal intersections.  From Gompf's isotopy lemma \cite[Lemma 5.3]{Ev11}, $\mC^0_\w$ is homotopic equivalent to $ \mC_\w$. Taking a filling divisor $\Sigma \in \mC^0_\w$, we consider the following iterated fibration.  

% , and we fix a base point $\Sigma\in\mC$. 
%  We denote by $\Sigma_{1}$ the component of $\Sigma$ with homology class $2H-E_1-\cdots-E_5$.

%\begin{equation}\label{summary}
%\begin{CD}
%Symp_c(U) = Stab^1(\Sigma) @>>> Stab^0(\Sigma) @>>> Stab(\Sigma) @>>> Symp_h(X, %\omega) \\
%@. @VVV @VVV @VVV \\
%@. \mG(\Sigma) @. Symp(\Sigma) @. \mC^0 \simeq \mJ_{\w}^{reg}
%\end{CD}
%\end{equation}
\begin{equation}\label{summary}
\begin{tikzcd}[column sep=small]
\Symp_c(X\backslash \Sigma) = Stab^1(\Sigma) \arrow[r] \arrow[d, phantom, ""{coordinate, name=Z}] &
Stab^0(\Sigma) \arrow[r] \arrow[d] &
Stab(\Sigma) \arrow[r] \arrow[d] &
\Symp_h(X, \omega) \arrow[d] \\
{} &
\mG(\Sigma) &
\Symp(\Sigma) &
\mC^0_\w \simeq \mJ_{\w}^{reg}
\end{tikzcd}
\end{equation}

Here, $\mJ_{\w}^{reg}$ is the subspace of the space of $\w$-compatible almost complex structures $\mJ_{\w}$, such that each class of irreducible components in $\Sigma$ has an embedded $J$-holomorphic representative. In the remainder of this section, when $\w$ is clear from the context, we will suppress this subscript for simplicity. 

From a standard codimension argument, $\mJ_{\w}^{reg}$ is path connected; hence, $\mC^0$ is also path connected.\footnote{This follows from a cobordism by a $1$-parameter family of almost complex structures, with an additional local adjustments at each positive intersections between different components of the configurations at all time.  See more details in \cite[Appendix A]{Ev11}.} Any isotopy between two elements of $\cC^0$ can be upgraded to an ambient isotopy by Banyaga's extension.  Therefore, $\cC^0$ is equivalent to the set of symplectic configurations which are ambiently isotopic to $\Sigma$.  Here is a glossary of the terms in \eqref{summary}:

\begin{itemize}
    \item $\mC^0$: the space of filling divisors whose components intersect symplectically orthogonally.\\
    \item $\Symp(\Sigma)$: the group of automorphisms of $\Sigma$ induced by $\Symp(X,\w)$.\\
    \item $Stab(\Sigma)$: the subgroup of $\Symp_h(X, \w)$ that preserves a given embedded configuration $\Sigma$.\\
    \item $Stab^0(\Sigma)$: the subgroup of $Stab(\Sigma)$ that fixes $\Sigma$ point-wise.\\
    \item $\mG(\Sigma)$: the gauge group of the normal bundle of $\Sigma$, i.e., automorphisms of the normal bundle of $\Sigma$ induced by $Stab^0(\Sigma)$.\\
    \item $Stab^1(\Sigma)$: the subgroup of $Stab^0(\Sigma)$ that fixes the normal bundle of $\Sigma$.\\
    % \item $Symp_c(U)$: the compactly supported symplectomorphisms of $U = X \setminus \Sigma$.\\
\end{itemize}

Each vertical map in \eqref{summary} is the fibration of a transitive action of a Lie group, and the horizontal maps are the inclusion of the corresponding isotropy group.   In addition, for some choice of $\lambda$, the above iterated fibrations behave particularly well:

\begin{itemize}
    \item For $n\ge6$,  when $\lambda$ satisfies
\begin{equation}\label{e:lambdaineq}
   \lambda > \frac{n-3}{n-1}, \quad \lambda \in \Q.
\end{equation}
The complement of $\Sigma$ is convex under this assumption, so one may prove that the compactly supported symplectomorphism group of $X \setminus \Sigma$ is contractible; see \cite[Lemma 5.5, 5.6]{LLW22}. 
Note that this condition forces positivity for any $\omega$ of type $\DD$.

  \item For $n=5$ and all possible 
\begin{equation}
    \lambda\in[{1}/{3}, 1]\cap \Q
\end{equation}
for which $\w$ is symplectic, $Symp_c(X\setminus\Sigma)\sim \Z$ \cite[Proposition 3.8]{LLWsmall}.
\end{itemize}

As a consequence, we have the following assertion.

\begin{proposition}\label{p:stab0ham}
   Let $(X,\w)$ be of pure type $\DD$. Suppose one of the following holds:
   
   \begin{itemize}
       \item $n\ge6$, $\lambda$ satisfies \eqref{e:lambdaineq}. 
       \item  $n=5$, $\lambda\in\Q$
   \end{itemize} 
   Then the map $\pi_0(Stab^0(\Sigma)) \to \pi_0(Stab(\Sigma))$ is trivial.
   Therefore, if $f \in Stab^0(\Sigma)$, then $f$ is Hamiltonian isotopic to the identity.    
\end{proposition}

\begin{proof}
For $n=5$, this was proved in  \cite[Lemma 4.1]{Ev11} and  \cite[Lemma 3.9]{LLWsmall}; see also \cite[Lemma 3.18]{LLWsmall}.  For $n\ge6$, this property has also been explained in the discussion above \cite[Lemma 5.7]{LLW22}. For the sake of the reader's convenience, we go through the logical steps of the proof without going into the calculations.  

When $\lambda$ satisfies \eqref{e:lambdaineq}, $Symp_c(X\setminus\Sigma)$ is weakly contractible and the fibration sequence becomes:
%\begin{equation}\label{summary2}
%\begin{CD}
% \{pt\} @>>> Stab^0(\Sigma) @>>> Stab(\Sigma) @>>> Symp_h(X, \omega) \\
%@. @VVV @VVV @VVV \\
%@. \mathbb{Z}^{2n-7} @. (S^1)^{2n-6} \times \Diff^+(S^2,n-1) @. \mC^0 \simeq \mJ_{\w,\Sigma}
%\end{CD}
%\end{equation}

\begin{equation}\label{summary2}
\begin{tikzcd}[column sep=small]
\{pt\} \arrow[r] &
Stab^0(\Sigma) \arrow[r] \arrow[d] &
Stab(\Sigma) \arrow[r] \arrow[d] &
Symp_h(X, \omega) \arrow[d] \\
{} &
\mathbb{Z}^{2n-7} &
(S^1)^{2n-6} \times \Diff^+(S^2,n-1) &
\mC^0 \simeq \mJ_{\w}^{reg}
\end{tikzcd}
\end{equation}

See \cite[Equation (31)]{LLW22}. Therefore, $Stab^0(\Sigma)$ is homotopic to $\mathbb{Z}^{2n-7}$. By \cite[Lemma 2.9]{LLWnle4}, the map $\pi_1((S^1)^{2n-6} \times \Diff^+(S^2,n-1)) \to \pi_0(Stab^0(\Sigma))=\mathbb{Z}^{2n-7}$ is surjective.
It implies that $\pi_0(Stab^0(\Sigma)) \to \pi_0(Stab(\Sigma))$ is the zero map, by the homotopy exact sequence.

%From \cite[5.3]{LLW22}, we have a composition of two maps:
%\[
%\pi_0(Stab(\Sigma)) \xrightarrow{\sim} \pi_0(\Diff^+(S^2, n-1)) \twoheadrightarrow \pi_0(\Symp_h(X, \w)).
%\]
%The first map is given by restricting the action of an element $g \in Stab(\Sigma)$ to the component of homology class $2H - E_1 - \cdots - E_5$, where the $n-1$ marked points are given by the intersections with $E_2, \ldots, E_5, H - E_1 - E_6, \ldots, H - E_1 - E_n$, respectively. The definition of the second map factors through the composition of the above isomorphism and the inclusion-induced homomorphism $\pi_0(Stab(\Sigma)) \xrightarrow{\sim} \pi_0(\Symp_h(X, \w))$. Therefore, if $g \in Stab^0(\Sigma)$, its image in $\Diff(S^2, n-1)$ must be trivial, which proves the claim.
\end{proof}

% ***07/31: Add n=5 case tomorrow below

The group $\Diff^+(S^2,n-1)$ in the proof of Proposition \ref{p:stab0ham} is the group of orientation-preservation diffeomorphisms of a sphere with $n-1$ marked points such that the marked points are {\bf pointwise-fixed}.
We have a fibration 
\begin{equation}\label{eq:DiffSeq}
\begin{tikzcd}[column sep=small]
\Diff^+(S^2,n) \arrow[r] &
\Diff^+(S^2) \arrow[d] \\
{} &
(S^2)^n\setminus \Delta
\end{tikzcd}
\end{equation}
where $(S^2)^n\setminus \Delta$ consists  of $n$ ordered pairwise distinct points on $S^2$, $\Diff^+(S^2)$ acts on $(S^2)^n\setminus \Delta$ transitively and the stablizer is  $\Diff^+(S^2,n)$.
We have the homotopy exact sequence $\pi_1(\Diff^+(S^2)) \to \pi_1((S^2)^n\setminus \Delta) \to \pi_0(\Diff^+(S^2,n)) \to  \pi_0((S^2)^n\setminus \Delta)=0,$ 
and it is well-known that the first arrow is injective.
This gives us 
\[
\pi_0(\Diff^+(S^2,n))=\pi_1((S^2)^n\setminus \Delta)/\pi_1(\Diff^+(S^2))=PB_{n}(S^2)/(\mathbb{Z}/2).
\]
The following property of (the quotient of the) pure braid groups plays a crucial role in our argument. 

\begin{lemma}[Bihopfian property, see \cite{LLWsmall} Lemma 3.4 and \cite{LLW22} Lemma 5.2]\label{l:bihopfian}
    Let $G=PB_{n}(S^2)/(\mathbb{Z}/2)=\pi_0(\Diff^+(S^2,n))$. Any group epimorphism or monomorphism $G \to G$ is an isomorphism.
\end{lemma}

% \subsection{Non-pure type $\DD$ forms}

% Suppose 

% {\color{blue} [need to fix notations, $\mJ_{\w}^{reg}$ is called $J^{reg}$ below, also the terms in the fibration sequence needed to be fixed]}

% {\color{blue} [where should we talk about the fibre class and ruling?]}

% {\color{blue} [don't need to define $\cA$ and $\cS$]}

% {\color{blue} [where to introduce inflation?]} 

% {\color{blue} [where to state that we can make a family of divisor $\omega$-orthogonal?]}

\section{$Stab^0(\Sigma)$ is in $\Ham$}
\label{sec:Stab0}

In this section, we generalize Proposition \ref{p:stab0ham} to all positive symplectic forms of type $\DD$, regardless of assumption \ref{e:lambdaineq} and purity. More precisely, the rest of the section is devoted to the proof of the following statement:

\begin{theorem}\label{t:stab0ham}
   Let $(X,\omega)$ be a positive rational surface of pure type $\DD$, and $\Sigma$ be a filling divisor. If $g \in Stab^0(\Sigma)$, then $g$ is Hamiltonian isotopic to the identity.    
\end{theorem}

The majority of the proof is devoted to the case where $X=X_n$ is of pure type $\mathbb{D}$; in fact, in this case it is of type $\DD_k$, with $k=n-1$. The case of $n=5$ follows from exactly the same argument and will be briefly discussed in Section \ref{sec:pure5}.

\subsection{Homotopy exact sequences and a diagram chasing}

% Let $\w_\lambda$ be a pure positive type $\DD$ form,
%    \begin{equation*}
%          [\w_{\lambda}]=(1|\lambda,\underbrace{\frac{1-\lambda}{2},\cdots,\frac{1-\lambda}{2}}_{n-1}),
%     \end{equation*}
% and $X$ be diffeomorphic to $\CP^2\#n\ov\CP^2$ for some $n \ge 6$. \textcolor{red}{[We need a remark about $n=5$ at some point.]}
%We fix $n$ in this section but change $\lambda$.

The fibration diagram \eqref{summary2} gives us homotopy exact sequences.
In particular, we are interested in the following diagram with exact rows.

\begin{center}
\begin{tikzcd}
 \pi_1(\cC^0_{\Sigma,\w_{\lambda}}) \arrow{r}{u}  & \pi_0(Stab(\Sigma,\w_{\lambda})) \arrow{r}{v}\arrow{d}{=} &\pi_0(\Symp_h(X,\omega_{\lambda})) \cong PB_{n-1}(S^2)/(\Z/2)   \\
  \pi_0(Stab^0(\Sigma,\w_{\lambda})) \arrow{r}{u'} & \pi_0(Stab(\Sigma,\w_{\lambda})) \arrow{r}{v'} & \pi_0(Aut(S^2,n-1))\cong PB_{n-1}(S^2)/(\Z/2) 
 \end{tikzcd}
 \end{center}
The maps $u,v,u',v'$ depend on $\omega_{\lambda}$, but we don't indicate it in the notations to lighten the notation.  The top-right isomorphism follows from \cite{LLW22}.

\begin{proposition}\label{p:bihopfian}
Fix $\omega_\lambda$, then the following two conditions are equivalent

\begin{itemize}
    \item $v'\circ u$ equals to $0$,
    \item $v \circ u'$ equals to $0$.
\end{itemize}  

If either holds, every element in $Stab^0(\Sigma,\w_{\lambda})$ is a Hamiltonian diffeomorphism.
\end{proposition}

\begin{proof}
Suppose that $v'\circ u=0$. Then $\im(u')\supset \im(u)$.  
Therefore, we have the quotient map 
\[
\pi_0(Stab(\Sigma,\w_{\lambda}))/\im(u)\to \pi_0(Stab(\Sigma,\w_{\lambda}))/\im(u')
\]  
By the exactness of the sequence, both the domain and target of this map are non-canonically isomorphic to $\pi_0(Aut(S^2,n-1))\cong PB_{n-1}(S^2)/(\Z/2)$.

The bihopfian property implies that this quotient map is an isomorphism.  Therefore, $\im(u)=\im(u')$, and that $\im(v\circ u')=v(\im(u))=0$.  Since the inclusion-induced map
\[\pi_0(Stab^0(\Sigma,\w_{\lambda}))\to\pi_0(Symp_h(X,\w_{\lambda}))\]
is exactly the map $v\circ u'$, it is trivial and hence every element in $Stab^0(\Sigma,\w_{\lambda})$ is a Hamiltonian diffeomorphism.

Conversely, if $v \circ u'=0$, then the same argument with the roles of $u$ and $u'$ exchanged shows that $v' \circ u=0$.
\end{proof}

\begin{lemma}\label{l:small}
If $\lambda$ satisfies \eqref{e:lambdaineq}, then the map $\pi_1(\cC^0_{\Sigma,\w_{\lambda}})\xrightarrow{v'\circ u}\pi_0(Aut(S^2,n-1))$ is trivial.
\end{lemma}

\begin{proof}
By Proposition \ref{p:stab0ham}, the map $v\circ u'$ is trivial when $\lambda$ satisfies \eqref{e:lambdaineq}. Therefore, $v'\circ u=0$ by Proposition \ref{p:bihopfian}.
\end{proof}

\subsection{An inflation argument}

Fix a type $\DD$ symplectic form $\w=\w_\lambda$.
We choose a $\hat\lambda$ such that it satisfies \eqref{e:lambdaineq} and $\lambda < \hat\lambda<1$. Consider a corresponding pure type $\DD$ form $\omega_{\hat\lambda}=:\hat\w$. 
Consider the corresponding maps 

\begin{equation}\label{eq:vu}
v'_{\w} \circ u_{\w}: \pi_1(\cC^0_{\Sigma,\w}) \to \pi_0(Aut(S^2,n-1))    
\end{equation}

and 
\begin{equation}\label{eq:vu'}
    v'_{\hat\w} \circ u_{\hat\w}: \pi_1(\cC^0_{\Sigma,\hat\w}) \to \pi_0(Aut(S^2,n-1)).
\end{equation}

The goal of this subsection is to use an inflation argument together with the triviality of 
$v'_{\hat\w} \circ u_{\hat\w}$ (i.e. Lemma \ref{l:small}) to show that $v'_{\w} \circ u_{\w}$ is trivial as well.

Let $I:=[0,1]$ and $\gamma: I \to \cC^0_{\w}$ satisfy $\gamma(0)=\gamma(1)=\Sigma$, so it represents a loop in $\pi_1(\cC^0_{\Sigma,\w})$\footnote{We use the domain $[0,1]$ rather than $S^1$ because later on we will consider some paths that do not close up and we want to use the same domain for comparison.}. 
There exists a path of almost complex structures $J_{\gamma}:[0,1]\to \cJ_\w$ such that $\gamma(t)$ is $J_{\gamma}(t)$-holomorphic for all $t \in I$ and $J_\gamma(0)=J_\gamma(1)$. 
By definition, $J_{\gamma}$ is a loop  in $\cJ_\w^{reg} \subset \cJ_\w$.
Recall that the unique $J_{\gamma}(t)$-holomorphic embedded representative of the class $H-E_1-E_6$ and $E_6$ are components of $\gamma(t)$. 
Denote the union of these two $J_{\gamma}(t)$-holomorphic embedded spheres by $F_t$.
The homology class of $F_t$ is $(H-E_1-E_6)+E_6=H-E_1$, which is the fibre class.

We apply the negative inflation to the two irreducible components of $F_t$ to get a family of symplectic forms, following Li-Usher \cite{LU06}, Buse \cite{Buse11}, and McDuff \cite{McDuffErratum}.  By inflating $E_6$ and $(H-E_1-E_6)$ alternately, we obtain an inflation of the symplectic class along the direction of the fiber class $H-E_1$.  See \cite[Lemma 3.21]{LLW22} for the general case when one inflates along multiple exceptional classes represented by embedded curves.

This leads to a family of symplectic forms $(\omega_{t,s})_{t \in I, s \in \mathbb{R}_{\ge 0}}$ such that $\omega_{t,0}=\omega$ for all $t \in I$, $\omega_{0,s}=\omega_{1,s}$ and
\begin{align*}
[\omega_{t,s}]&=[\omega_0]+s[F_t]\\
&=(H-\lambda E_1-\sum_{i=2}^n \frac{1-\lambda}{2}E_i)+s(H-E_1)\\
&=(1+s)H-(\lambda+s)E_1-\sum_{i=2}^n \frac{1-\lambda}{2}E_i\\
&=(1+s)\left(H-\frac{\lambda+s}{1+s}E_1-\sum_{i=2}^n \frac{1-\lambda}{2(1+s)}E_i\right)\\
&=(1+s)\left(H-\frac{\lambda+s}{1+s}E_1-\sum_{i=2}^n \frac{1-(\frac{\lambda+s}{1+s})}{2}E_i\right).
\end{align*}
Notice that $\frac{1}{1+s}\omega_{t,s}$ is of pure type $\DD$.
Pick $s_0$ such that $\frac{\lambda+s_0}{1+s_0}=:\hat\lambda$ satisfies \eqref{e:lambdaineq} and denote the pure type $\DD$ form $\frac{1}{1+s_0}\omega_{t,s_0}$ by $\w'_t$. 

% ***June 18 Weiwei

The configurations of surfaces $\gamma(t)$, now viewed as inside $(X,\w'_t)$, give us a family of symplectic divisors $\gamma'(t)$ in $(X,\w'_t)$, because every component of $
\gamma(t)$ intersects $\w$ orthogonally with $F_t$ or is contained in $F_t$.%\textcolor{red}{[add McDuff-Opshtein \cite[Section 5]{MO15} as a reference]} 

Let $R:=[0,1] \times [0,1]$, and we define a family of symplectic forms parametrized by $\partial R$ as follows.

\[
\omega'_r = 
\begin{cases}
\omega_t', & \text{if } r = (t,1), \\
\omega_0', & \text{otherwise } r\in\partial R.
\end{cases}
\]

By abuse of notation, we continue to use $\w'_t$ to denote $\w'_{(t,1)}$ later.  We also define the same family of $\gamma'(r)$ for $r\in\partial R$ and denote $\gamma'(t)$ as $\gamma'((t,1))$.  We next show that $\w'_r$ extends to the interior of $R$.

% $[0,1] \times \{1\}$ with $I$ by the obvious map.
%  We can think of the loop $\w'_t$ and $\gamma'(t)$ as an $I$-family of symplectic forms and symplectic divisors, respectively.
%  Since $\w'_0=\w'_1$ and $\gamma'(0)=\gamma'(1)$, we can extend them trivially to an $\partial R$-family of symplectic forms and symplectic divisors, respectively, by defining them as constant in $\partial R \setminus I$.
% We will write the symplectic form as $\w'_r$ if we want to use the parameter space $r \in \partial R$, and write it as $\w'_t$ if we want to use the parameter space $t \in I$.  Similarly for $\gamma'$.

\begin{lemma}\label{l:Rfamily}
    There is an $R$-family of cohomologous symplectic forms $(\omega'_r)_{r \in R}$ on $X$, which extends $(\w'_r)_{r \in \partial R}$.
\end{lemma}

\begin{proof}
    We can extend $J_\gamma$ to an $R$-family $J_{\gamma}:R \to \cJ_\w$ (not $\cJ_\w^{reg}$) because $\cJ_\w$ is contractible.
Since $H-E_1-E_6$ and $E_6$ are exceptional classes of the minimal area, for every $r \in R$, there is a unique embedded $J_{\gamma}(r)$-holomorphic representative.
By possibly perturbing $J_{\gamma}(r)$, we can assume that the $J_{\gamma}(r)$-holomorphic representative of $H-E_1-E_6$ and $E_6$ intersect orthogonally. We can then do a family inflation for $r \in R$ along these representatives as in the case when $r\in\partial R$. The result is an $R$ family of cohomologous symplectic forms as claimed. 
\end{proof}

Note that Lemma \ref{l:Rfamily} does not contradict the nontriviality of $\gamma(t)$ in general: when $J_{\gamma}(r) \notin \cJ_\w^{reg}$, we cannot find a corresponding $J_{\gamma}(r)$-holomorphic configuration in $\cC^0_{\Sigma,\w}$.  In the following lemma, we denote $r=(x,y) \in R$.

\begin{lemma}[Moser argument]
    There is a smooth family of diffeomorphisms $(\phi_{(x,y)})_{(x,y) \in R}$ such that $\phi_{(x,0)}=\phi_{(0,y)}=\phi_{(1,y)}=\id$ and $\phi_{(x,y)}^*\w'_{(x,y)}=\w'_{(x,0)}=\w'_{(0,0)}$ for all $(x,y) \in R$.

\end{lemma}

\begin{proof}
We can construct an $R$-family of $1$-forms $\alpha_{(x,y)}$ such that $d\alpha_{(x,y)}=\frac{d}{dy} \omega_{(x,y)}$
such that $\alpha_{(0,y)}=0=\alpha_{(1,y)}$ for all $x, y \in [0,1]$ by \cite[Section 3.2]{MSIntro}.
We then run the Moser argument using the family $(\alpha_{(x,t)})_{t\in[0,1]}$ for every $x$ to get $\phi_{(x,y)}$.
\end{proof}

\begin{figure}
    \centering
    \includegraphics[scale=1.8, width=0.65\linewidth]{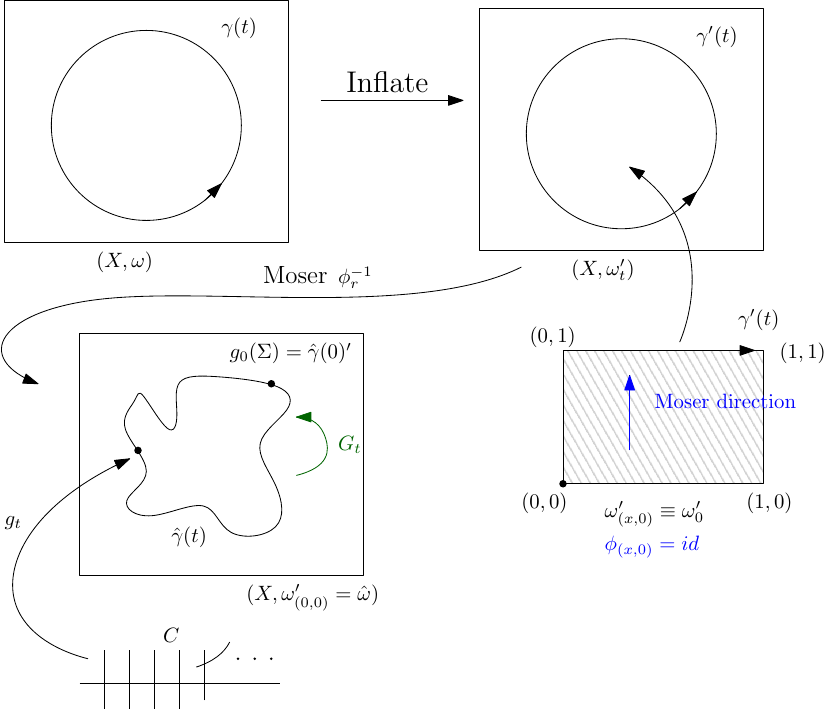}
    \caption{Construction of $\hat\gamma(t)$ from Moser techniques}
    \label{fig:Moser}
\end{figure}

\begin{corollary}\label{c:boundaryLoop}
The loop    $\phi_{(x,y)}^{-1}(\gamma'(x,y))$  for  $(x,y) \in \partial R$ is a loop of symplectic divisors in $(X,\w'_{(0,0)})$.  
\end{corollary}
% this loop
We define $\hat\gamma(r):=\phi_{r}^{-1}(\gamma'(r))$.
Note that $\hat\gamma(x,0)=\hat\gamma(0,y)=\hat\gamma(1,y)=\gamma'(0,0)$ for all $x,y \in [0,1]$ because $\phi_{(x,0)}=\phi_{(0,y)}=\phi_{(1,y)}=\id$.
Therefore, $\hat\gamma(x,1)$ for $x \in [0,1]$ is a loop of symplectic divisors as well (i.e. $\hat\gamma(0,1)=\hat\gamma(1,1)$).
We write $\hat\gamma$ as $\hat\gamma(t)$ when we restrict the family to $I = [0,1] \times \{1\} \subset \partial R$. 
Let $\hat\w:=\w'_{(0,0)}$, then Corollary \ref{c:boundaryLoop} implies that $\hat\gamma(t)$ represents an element in $\pi_1(\cC^0_{\Sigma,\hat\w})$; see Figure \ref{fig:Moser}.

% Recall that we have defined the maps $v'_{\w} \circ u_{\w}:\pi_1(\cC^0_{\Sigma,\w}) \to \pi_0(Aut(S^2,n-1))$ and $v'_{\w'} \circ u_{\w'}:\pi_1(\cC^0_{\Sigma,\w'}) \to \pi_0(Aut(S^2,n-1))$.

We are now ready to analyze maps \eqref{eq:vu} and \eqref{eq:vu'}. Describe $v'_{\hat\w} \circ u_{\hat\w}(\hat\gamma)$ explicitly as follows.  Consider a parametrization of $\hat\gamma(t)$, that is, choose a nodal symplectic curve $C$ and a family of embeddings
 $g_t: C \to (X,\hat\w)$, such that 
\begin{equation}
    g_t(C)=\hat\gamma(t), \quad t \in [0,1],
\end{equation}
 
  where $g_t$ is smooth on every irreducible component of $C$.  Then $g_t$ forms a symplectic isotopy from $\hat\gamma(0)$ to itself.
  
% In particular, $g_1: \hat\gamma(0) \to \hat\gamma(0)$ is a symplectomorphism on every irreducible component of $\hat\gamma(0)$.

By the Banyaga extension theorem, we can extend $g_t$ to a path of Hamiltonian diffeomorphisms $G_t:(X,\hat\w) \to (X,\hat\w)$ based at the identity.
By definition, $u_{\hat\w}(\hat\gamma)=[G_1] \in \pi_0(Stab(\Sigma,\hat\w))$.
Denote the irreducible component of $\hat\gamma(0)$ in the class $2H-E_1-\cdots-E_5$ by $\hat\gamma(0)_0$, and the corresponding irreducible component in $C$ as $C_0$.
Then, by definition, $v'_{\hat\w} \circ u_{\hat\w}(\hat\gamma)=[G_1|_{\hat\gamma(0)_0}]=[g_0^{-1}\circ g_1|_{C_0}] \in \pi_0(\Diff^+(S^2,n-1))$.
By Lemma \ref{l:small}, we know that $v'_{\hat\w} \circ u_{\hat\w}$ is trivial, hence so is $[g_0^{-1}\circ g_1|_{C_0}]$.

Now consider the symplectomorphisms $\phi_t \circ G_{t}: (X,\hat\w) \to (X,\w'_t)$ for $t \in I$.
By definition, 
\[
\phi_t \circ G_t(\hat\gamma(0))=\phi_t \circ g_t(C)=\phi_t(\hat\gamma(t))=\gamma'(t).
\]
% and 
% \[
% \phi_{(0,0)} \circ G_1|_{\hat\gamma(0)_0}=G_1|_{\hat\gamma(0)_0}
% \]
% because $\phi_{(0,0)}=\id$.

As subsets, we have $\gamma'(t)=\gamma(t)=\phi_t\circ g_0(C)$ in $X$ (they only differ by thinking of them as living in $X$ with different symplectic forms); therefore, it makes sense to ask whether $\phi_t\circ g_0: C\to (X,\w)$ is a also symplectic isotopy.
Indeed, on every irreducible component of $C$ other than $C_0$, $\phi_t \circ g_t$ is a symplectic isotopy with respect to $(X,\w)$ because each of these other components either do not intersect $F_t$ or are contained in $F_t$ and we have $\w'_t=\frac{1}{1+s_0}\w$ outside a small neighborhood of $F_t$, and also when restricted on $F_t$.
%In other words, on every irreducible component of $\hat{\gamma}(0)$ other than $\hat{\gamma}(0)_0$, $\phi_r \circ g_r$ is still a symplectic isotopy.

On the other hand, $\phi_t \circ g_t|_{\hat{\gamma}(0)_0}$ is not necessarily a symplectic isotopy with respect to $(X,\w)$.
However, the family of symplectic forms $((\phi_t \circ g_t)^*\w)_{t \in [0,1]}$ on $C_0$ are cohomologous and independent of $t$ near the intersection points between $C_0$ and other components of $C$.
Therefore, by the Moser argument, we can find a family of diffeomorphisms $d_t: C_0 \to C_0$, starting from $d_0=\id$ and supported away from the intersection points with other components of $C$, such that $\phi_t \circ g_t \circ d_t:C_0 \to (X,\w)$ is a symplectic isotopy.
By Banyaga extension and restriction again, we know that 
$v'_{\w} \circ u_{\w}(\hat{\gamma})=[\phi_1 \circ g_1 \circ d_1|_{C_0}] \in \pi_0(\Diff^+(S^2,n-1))$.
Since $\phi_1=\id$ and $d_1$ is isotopic to $d_0=\id$ in the complement of the marked points,  $[\phi_1 \circ g_1 \circ d_1|_{C_0}]=[g_1|_{C_0}]$ is trivial.

Therefore, we showed that $v'_\w\circ u'_\w(\hat\gamma)=0$.  From Proposition \ref{p:bihopfian}, this concludes the proof of Theorem \ref{t:stab0ham} for $n\ge6$.

\subsection{Pure type $\mathbb D$ case for $n=5$}\label{sec:pure5}

 There are two cases left behind in the previous section.  When $\lambda\neq{1}/{3}$, i.e. pure type $\DD_4$, although the filling divisor is different, we draw the same conclusion by exactly the same inflation argument since Proposition \ref{p:stab0ham} already covers all cases when $\lambda\in\Q$.  When $\lambda={1}/{3}$, i.e. pure type $\DD_5$, this is a special case of Proposition~\ref{p:stab0ham}.

\section{$C^0$ small symplectomorphisms cannot braid}

In this section, we prove the following result.

\begin{theorem}\label{t:trivial_braid}
Let $(X,\w)$ be of type $\DD$ and $\Sigma \subset X$ a filling divisor. 
There exists $\epsilon>0$ such that if $d_{C^0}(f,\id)<\epsilon$, then there is a Hamiltonian diffeomorphism $\phi$ such that $\phi \circ f \in Stab^0(\Sigma)$.

%{\cym [old: Suppose $(f_n)_{n\in \mathbb{Z}}$ is a sequence of symplectomorphisms $C^0$ converging to the identity. Then there is $N>0$ such that for all $n>N$, there is a Hamiltonian diffeomorphism $\phi_n$ such that $\phi_n \circ f_n \in Stab^0(\Sigma)$. ]}
\end{theorem}

%\textcolor{red}{[Do we need to assume type $\DD$ for this theorem? Also I think we should change the statement to "there exists $\epsilon>0$ such that if $d_{C^0}(f,\id)<\epsilon$ then..."]}\textcolor{blue}{1. Yes, I agree this could be a better statement.  If you decide to change it, please check the consistency in the proof of Theorem 1.2 (the end of Section 4.) 2. We do rely on type $\mathbb{D}$ assumption in the following aspects: 1. the context of $Stab(\Sigma)$, because we haven't defined this notion for other cases; 2. the concrete form of the total space, which is a ruled surface.  This also works for type $\mathbb{A}$ but not $\mathbb{E}$, and the case of $\mathbb{A}$ is trivial from the beginning.}{\cym [it looks to me that it is consistent with  the proof of Theorem 1.2 ]}

The proof of Theorem \ref{t:trivial_braid} is again divided into several cases, and the main case remains the pure type $\DD$ case for $n\ge6$  addressed in Section \ref{sec:mainproof1}, while the non-pure case and the case $n=5$ are considered in Section \ref{sec:mainproof2}.

Throughout this section, we consider a class of configurations of curves (and their corresponding homology classes) slightly different from the filling divisors we considered in previous sections.  

For any $J\in\cJ^\reg=\cJ^\reg_\omega$, we denote $D_{i,J}$ as the unique nodal curve with two components of homology classes $\{E_i, H-E_1-E_i\}$ for $i\ge2$, and $D_{1,J}$ as the $J$-holomorphic representative of class $2H-E_1-\cdots-E_5$ in this section.  We also denote

$$D_J:=\cup_{i\ge1}D_{i,J}.$$

\subsection{The case of pure type $\DD$, $n\ge6$}\label{sec:mainproof1}
\subsubsection{Pseudoholomorphic ruling fibration}

% Let $\Sigma$ be the configuration in $(X,\w)$ introduced in the previous section.

Let $F=H-E_1$ be the fiber class. 
For every $J \in \JJ^\reg$, let $\cM^{\circ}(F,J)$ be the moduli space of $J$-holomorphic spheres in the class $F$ modulo parametrization, and let $\cM(F,J)$ be its Gromov compactification.
Denote 
$$p_J:=\cM(F,J)\setminus \cM^{\circ}(F,J).$$ 

Observe that
$$p_J=\cup_i p_{i,J},$$

\noindent where $p_{i,J}$ represents $D_{i,J}$ in the moduli space for $i>1$.

 % We now turn to the special case where $J\in\JJ^\reg$.  Then every $[D_i]$ admits an embedded (necessarily unique) $J$-holomorphic representative, denoted by $D_{i,J}$.  
 
%  % For $J\in\JJ^\reg$, $D_{0,J}$ is also embedded.
% We also denote $\cup_{i=1}^k D_{i,J}$ by $D_J$.  In this case, $(\cM(F,J), p_J)$ is homeomorphic to a $2$-sphere with $k$ marked points \textcolor{red}{[Need to revisit to see if we need any almost complex structure here.]}.

Let
\[
\chi_J:X \to \cM(F,J)
\]
be the projection map that sends $x$ to the unique $F$-curve (possibly nodal) passing through it.  This gives a $J$-holomorphic fibration with fibre class equal to $F$.

Notice that the restriction 
\[\chi_J|_{D_{1,J}}: D_{1,J} \to \cM(F,J) \]
 is a homeomorphism, so we can identify $\cM(F,J)$ with $D_{1,J}$.  Under this identification, $p_{i,J}$  is identified with the intersections $$q_{i,J}:=D_{i,J}\cap D_{1,J}\quad q_J:=\cup_i q_{i,J}.$$  
 Combining the two maps above, we obtain the following projection.
\begin{equation}
    \pi_J: X\to D_{1,J},\quad D_{i,J}\mapsto q_{i,J}.
\end{equation}

\subsubsection{Pseudo-holomorphic isotopy}\label{ssec:StepAB}

% ***For convenience of exposition, we introduce an \textbf{extended filling divisor}, which is a symplectic divisor that contains a filling divisor as a subset, and also includes an exceptional curve in homology classes $\{H-E_1-E_i\}_{i=2}^5$.  Additionally, we require that this extended filling divisor be holomorphic with respect to some $J\in \JJ^\reg$. {\color{blue} [Shall we take $D$ to include also the remaining exceptional curves in the fibre? Easy for exposition in this section]}

% Fix a choice of extended filling divisor $D$ in $(X,\w)$ for the rest of the section.  Let $J_0$ be an $\w$-compatible almost complex structure such that $D$ is $J_0$-holomorphic. {\color{green}[These two paragraphs should be removed later]}

Fix $J_0\in\JJ^\reg$ so that the corresponding $D_{J_0}$ has symplectically orthogonal intersections between different irreducible components.  For simplicity, we will denote 
$$D:=D_{J_0},\quad D_i:=D_{i,J_0}$$
$$ p_i:=p_{i,J_0},\quad q_i=q_{i,J_0}.$$

In this section, we will assume that $f$ is a symplectic diffeomorphism such that
\begin{equation}\label{e:epsilon}
   d_{C^0}(f,\id)<\epsilon,    
\end{equation}
and prove that $f(D)$ can be Hamiltonian isotopic $D$ in a controlled way in Corollary \ref{c:A-extension} and Lemma \ref{l:HB}.  The choice of $\epsilon$ will be specified below.
% % 
% This will imply that $f$ is isotopic to the identity by Theorem \ref{p:stab0ham} since $D$ contains a filling divisor as a subset.

% Let $(f_n)_{n \in \mathbb{N}}$ be a sequence of symplectomorphisms that converges to the identity in $C^0$-topology.

% In the remainder of the proof, we let $f$ be one of the $f_n$ that we will require to be sufficiently $C^0$-close to the identity.  

\subsubsection*{Step A: moving $f(D_i)$ back to $D_i$ for $i>1$}

Since $f$ is a symplectomorphism, $J_0^f:=f_*J_0 \in \JJ^\reg$.
In fact, since $f$ is in the Torelli part (as explained in the introduction), $f(D_i)$ is the unique  $f_*J_0$-holomorphic representative of $[D_i]$.

Let $p_{i,J^f_0}\in\cM(F,J^f_0)$ denote the points that represent $D_{i, J^f_0}$ for $i>1$.
Let $U_{i}^f \subset \cM(F,J^f_0)$ be an open neighborhood of $p_{i,J^f_0}$ such that 
\begin{enumerate}
    \item $D_i \subset \chi_{J^f_0}^{-1}(U_{i}^f)$, and 
    \item $U_{i}^f \cap U_{j}^f= \emptyset$ if $i \neq j$
\end{enumerate}

This is possible because $f$ is $C^0$-small.  We define
\begin{equation}
    \widetilde{U}_{i}^f:=\chi_{J^f_0}^{-1}(U_{i}^f),
\end{equation}

% We order the marked points $p_{J_0}$ in $\cM(F,J_0)$ by $p_{i,J^a_0}$ using the ordering of $D_i$ for $i>0$.

On top of these, we also need some open sets in $\cM(F,J_0)$.  For each $i>1$, choose $V_{i} \subset W_{i} \subset \cM(F,J_0)$ which are small open disk neighborhoods of $p_{i,J_0}$ such that 
\begin{enumerate}
   \item $\widetilde{V}_{i}:=\chi_{J_0}^{-1}(V_{i}) \subset \widetilde{U}_{i}^f \subset \chi_{J_0}^{-1}(W_{i})=:\widetilde{W}_{i}$, and 
    \item $W_{i} \cap W_{j}= \emptyset$ if $i \neq j$
\end{enumerate}

For the existence of these subsets, we have taken advantage of the fact that $f$ is $C^0$-small. More precisely, we should first choose $W_i$, then choose $\epsilon$ to be sufficiently small so that some $U_i^f\Subset\wt W_i$, and then choose $V_i$ so that $\wt V_i\subset \wt U_i^f$.
Denote $U^f:=\cup_i U_{i}^f$ and $\widetilde{U}^f:=\cup_i \widetilde{U}^f_{i}$, and similarly for $V$ and $W$. 

% We will choose isotopies of symplectic divisors from isotopies of almost complex structures, adapted to these open sets.

% ***WWW July 25

\begin{construction}\label{construction}
    Let $(J^f_t)_{t \in [0,1]}$ be a path in $\JJ^\reg$ such that
\begin{enumerate}
    \item $J^f_0=f_*J_0$,  
    \item for all $t \in [0,1]$, $J^f_t=J^f_0$ outside $\widetilde{U}^f$, 
    \item when $t=1$, $J^f_1=J_0$ inside $\widetilde{V}$
\end{enumerate}
\end{construction}

 % (\cite[Lemma 3.13]{LLW22})
Such a path always exists since $\cJ\setminus\cJ^\reg$ is a union of submanifolds of codimension 2 or higher.  With this choice, the $J^f_t$-holomorphic foliation outside $\widetilde{U}^f$ is independent of $t$.
Moreover, at $t=1$, the $J^f_1$-holomorphic foliation inside $\widetilde{V}$ is the same as the $J_0$-holomorphic foliation inside $\widetilde{V}$.

%In terms of moduli spaces $\cM(F,J^a_t)$, there is an open subset $U_{J^a_t}$ of $\cM(F,J^a_t)$ that is canonically identified with $U$ for all $t$ \textcolor{red}{[Would be good to add a word about how $U_{J^a_t}$ is obtained]}\textcolor{blue}{[I suspect this might not be true.  There is a canonical identification of the complement of $U$, which can make $U_{J_t^a}$ be defined as the complement of this subset in respective moduli, but points for different $t$ are hardly identified naturally.]}. 
%At the same time, when $t=1$, there is an open subset $V_{J^a_1}$ of $\cM(F,J^a_1)$, which can be canonically identified with 
%$V$ of $\cM(F,J_0)$.

\begin{lemma}\label{l:pos}
For $i>0$,    $D_{i,J^f_t}$ is contained in $\widetilde{U}^f$ for all $t$.
\end{lemma}

\begin{proof}
    Assume the contrary.  If $x\in D_{i,J^f_t}\cap (\wt U^f)^c$ for some $t$, then the $J^f_t$-curve in class $F$ passing through $x$ must be contained entirely in $(\wt U^f)^c$ from Construction \ref{construction} (2).  Therefore, this curve is irreducible and intersects positively with $D_{i,J^f_t}$, which contradicts the fact that $F^2=0$.
\end{proof}

\begin{corollary}\label{c:A-extension}
There is an ambient Hamiltonian isotopy $\phi_{H_A}^t:X \to X$ supported in $\widetilde{U}^f$, for $t \in [0,1]$, such that
\begin{enumerate}
\item $\phi_{H_A}^0=\id$, and
\item $\phi_{H_A}^t(D_{i,J^f_0})=D_{i,J^f_t}$ for all $i>0$ and for all $t \in [0,1]$
\end{enumerate}
\end{corollary}

\begin{proof}
Apply Banyaga extension to the symplectic isotopy $D_{i,J_t^f}$, which is guaranteed to be supported inside $\wt U^f$ by Lemma \ref{l:pos}.
 % \textcolor{red}{[Could you also add a little bit more please?]}
\end{proof}

% We have 
% $\phi_{H_A}^t(D_{i,J^f_0})=D_{i,J^A_t}$.  Therefore, for $i>0$, we have $\phi_{H_A}^1(D_{i,J^f_0})=D_{i,J^A_1}=D_{i,J^f_1}=D_{i,J_0}$.  The first equality follows from the definition, the second from Corollary \ref{c:A-extension}, and the last from Construction \ref{construction} (3).  So  $\phi_{H_A}^t$ is the desired isotopy.

\begin{figure}
    \centering
    \includegraphics[scale=2, width=0.8\linewidth]{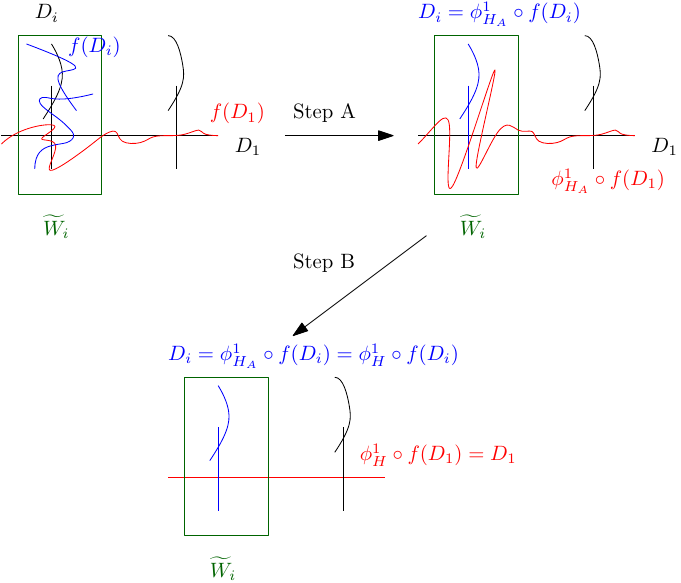}
    \caption{Isotopy of $D_i$ }
    \label{fig:Moser}
\end{figure}

\subsubsection{Step B: moving the entire configuration back to $D$}

Now define $J^A:=(\phi_{H_A}^1)_*J^f_0$.  Let $(J^B_t)_{t \in [0,1]}$ be a path in $\JJ^{\reg}$ such that
\begin{enumerate}
    \item when $t=0$, we have $J^B_0=J^A$, and 
    \item for all $t \in [0,1]$ and all $i>1$, $D_{i,J^B_0}=D_{i,J^A}$ is $J^B_t$-holomorphic (i.e. $D_{i,J^B_t}=D_{i,J^B_0}$), and
    \item when $t=1$, we have $J^B_1=J_0$.
\end{enumerate}
With this choice, we have a unique embedded $J^B_t$-holomorphic representative $D_{i,J^B_t}$ for all $t$ and all $i \ge 1$.  Moreover, $D_{i,J^B_t}=D_{i,J_0}$ for $i>1$ and for all $t$.  By perturbing the family $D_{1,J^B_t}$ to $D_{1,J_t^B}'$, we can obtain an $\omega$-orthogonal family 
\[
D_{J_t^B}':=D_{1,t}' \cup \bigcup_{i>1} D_{i,J^B_t}.
\]

% In terms of the moduli spaces, it means that the marked points of $\cM(F,J^B_t)$ are canonically identified with 
% the marked points of $\cM(F,J_0)$ for all $t$.

% Since $J^B_t  \in \JJ^\reg$ for all $t$, we have a unique embedded $J^B_t$-holomorphic representative $D_{i,J^B_t}$ for all $t$ and all $i \ge 1$.

Since $D_{1,J_0^B}=D_{1,J^A}$ is orthogonal to $D_{i,J^A}$, the perturbation can be chosen trivially when $t=0$, i.e., $D'_{1,J_{0}^B}=D_{1,J_0^B}$.
The following lemma is again a direct consequence of Banyaga's extension theorem applied to $D_{J^B_t}'$.

\begin{lemma}\label{l:HB}
There is an ambient Hamiltonian isotopy $\phi_{H_B}^t:X \to X$ for $t \in [0,1]$, such that
\begin{enumerate}
\item $\phi_{H_B}^0=\id$, and
\item $\phi_{H_B}^t(D_{J_0^B}')=D_{J_t^B}'$ for all $t \in [0,1]$
\end{enumerate}
\end{lemma}

Note that $D_{J_1^B}'=D_{J_0}=D$, so we have moved the whole configuration $D_{J_0^f}$ back to $D$.

\begin{figure}
    \centering
    \includegraphics[scale=0.85]{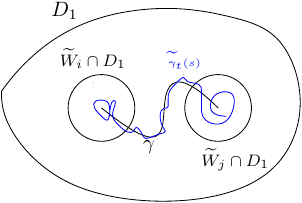}
    \caption{Recognizing the mapping class of $\gamma$-projection}
    \label{fig:Moser}
\end{figure}

\subsubsection{Trivial mapping class}\label{ssec:trivialMC}

Let $\phi_H^t$ be the concatenation of $\phi_{H_A}^{t}$ and $ \phi_{H_B}^{t}$ ($\phi_{H_A}^{t}$ goes first).
%Let $J^c_t$ be the concatenation of $J^a_{2t}$ and $J^b_{2t}$.
The goal of this section is to show that
\[
\phi_H^1 \circ f|_{D_1}: (D_1, q_{J_0}) \to (D_1, q_{J_0})
\]
represents the trivial mapping class in $\pi_0(\Diff(D_1, q_{J_0}))$.

\begin{lemma}\label{l:bij}
    For any $t \in [0,1]$ and $i$, we have the following equality
    \begin{equation}
        \pi_{J_0} \circ \phi_{H_B}^t \circ \phi_{H_A}^1 \circ f(q_{i})=q_{i}.
    \end{equation} 

\end{lemma}

\begin{proof}
   This follows from the definition of $\phi_{H_A}^1$, which sends $D_{i,J_0^f}$ to $D_{i}=D_{i,J_0}$ for $i>1$. For any $i>1$, since $\phi_{H_B}^t \circ \phi_{H_A}^1 \circ f(q_i)=D_{1,J_t^B}\cap D_{i,J_t^B}=D_{1,J_t^B}\cap D_{i}$, the $\pi_{J_0}$-images of it is $q_i$.
\end{proof}

Let $\Gamma$ be the space of continuous paths $\gamma:[0,1]  \to D_1$ such that $\gamma(0), \gamma(1) \in q_{J_0}$ and $\gamma(s) \notin q_{J_0}$ for $s \in (0,1)$.  We have the following elementary fact, for which the proof is omitted.

\begin{lemma}\label{l:curvecomplex}
Suppose that for any $\gamma \in \Gamma$,
$\phi_H^1 \circ f \circ \gamma$ is homotopic to $\gamma$ within the space $\Gamma$. Then
$\phi_H^1 \circ f|_{D_1} \in \Diff(D_1,q_{J_0})$ is in the trivial mapping class.
\end{lemma}

Lemma \ref{l:curvecomplex} and the following corollary make sense because of Lemma \ref{l:bij}, which implies that $\phi_H^1 \circ f \circ \gamma$ and $\pi_{J_0} \circ \phi_{H_A}^1 \circ f \circ \gamma$ are elements in $\Gamma$.

\begin{corollary}\label{c:homotopy_equal}
    Suppose that for any $\gamma \in \Gamma$,
$\pi_{J_0} \circ \phi_{H_A}^1 \circ f \circ \gamma$ is homotopic to $\gamma$ within the space $\Gamma$. Then
$\phi_H^1 \circ f \in \Symp_h(X)$ is in the trivial mapping class.
\end{corollary}

\begin{proof}
   % By Lemma \ref{l:bij}, $\pi_{J_0} \circ \phi_{H_B}^t \circ \phi_{H_A}^1 \circ f \circ \gamma$ for $t \in [0,1]$ defines a path of elements in $\Gamma$. 
   % Therefore, for any $\gamma \in \Gamma$, $\gamma$ is homotopic to $\pi_{J_0} \circ \phi_{H_A}^1 \circ f \circ \gamma=\pi_{J_0} \circ \phi_{H_B}^0 \circ \phi_{H_A}^1 \circ f \circ \gamma$
   % and hence homotopic to $\pi_{J_0} \circ \phi_{H_B}^1 \circ \phi_{H_A}^1 \circ f \circ \gamma=\phi_H^1 \circ f \circ \gamma$.
   
   % The result follows from Lemma \ref{l:curvecomplex} and Theorem \ref{t:stab0ham}.

   From Lemma \ref{l:bij}, $$\pi_{J_0}\circ\phi^1_{H_A}\circ f\circ\gamma\sim \pi_{J_0}\circ\phi_{H_B}\circ\phi_{H_A}\circ f\circ\gamma\sim \phi_{H}^1\circ f\circ \gamma$$
   as paths with fixed endpoints in $\Gamma$.  From Lemma \ref{l:curvecomplex}, the mapping class induced on $D_1$ is trivial, so the conclusion follows from Theorem \ref{t:stab0ham}.
\end{proof}

%Mak here 2nd Aug, 2025

Lastly, we verify the assumption of Lemma \ref{l:curvecomplex}.  Recall the open neighborhood $W$ we defined earlier.
Take $\gamma \in \Gamma$, we may assume without loss of generality (by possibly isotoping $\gamma$) that $\gamma(s) \in \wt V$ if and only if $s \in [0,1/3) \cup (2/3,1]$.
% Let $\widetilde{\gamma}:[0,1] \to X$ be such that $\pi_{J_0}(\widetilde{\gamma}(s))=\gamma(s)$ for all $s$.
Since $\widetilde{U} \subset \widetilde{W}$
and $\phi_{H_A}^t$ is supported in $\widetilde{U}$, we have the following assertions from the construction

\begin{proposition}\label{p:ABkey}
    \begin{enumerate}
    \item $\pi_{J_0}(\phi_{H_A}^t(f(\gamma(s)))) \subset \wt W\cap D_1$ for all $t \in [0,1]$ and $s \in [0,1/3) \cup (2/3,1]$, and 
    \item $\pi_{J_0}(\phi_{H_A}^t(f(\gamma(s))))=\pi_{J_0}(f(\gamma(s)))$ for all $t \in [0,1]$ and $s \in [1/3,2/3]$, and
    \item for $s=0,1$, we have $\pi_{J_0}(\phi_{H_A}^1(f(\gamma(s)))) =\gamma(s)\in p_{J_0}$.
\end{enumerate}

\end{proposition}

\begin{lemma}
The path $\pi_{J_0}(\phi^1_H(f(\gamma(s))))$ is in the same isotopy class as $\gamma(s)$ in $\Gamma$.
\end{lemma}

\begin{proof}
From items (1), (2), and (3) above, $\wt\gamma_t(s):=\pi_{J_0}\circ\phi_{H_A}^t\circ f(\gamma(s))$ has the following feature:

\begin{enumerate}[(i)]
    \item $d_{C^0}(\wt\gamma_t(s),\gamma(s))\le C\cdot d_{C^0}(f,\id)$ for $s\in[\frac{1}{3},\frac{2}{3}]$, where $C$ is a constant which depends only on $J_0$,
    \item $\wt\gamma_t(s)\subset \wt W\cap D_1$ for $s\in [0,\frac{1}{3}]\cup[\frac{2}{3}, 1]$.  
    \item $\wt\gamma_1(s)\notin p_{J_0}$ for $s\neq0,1$.
\end{enumerate}

Here, (i) follows from the fact that $\phi_{H_A}^t$ is supported inside $\wt U^f$ and (ii) follows from the fact that $f$ is $C^0$-small. To see (iii), note that  $\phi_{H_A}^1(f(D_i))=D_i$ for $i>1$, so the positivity of intersection demands $\phi_{H_A}^1(f(\gamma))\subset \phi_{H_A}^1(f(D_1))$ intersects $D_i$ at a unique point.
Note that $\wt W\cap D_1$ is homeomorphic to the union of $k$ disjoint smooth disks.  

Therefore, both $\gamma(s)$ and $\wt\gamma_1(s)$ are curves connecting $q_i$ and $q_j$ for some $i, j\in \{1,\cdots, k\}$.  When $\epsilon$ is chosen small enough in \eqref{e:epsilon}, we may first apply an isotopy for $\wt\gamma_t(s)|_{s\in[\frac{1}{3},\frac{2}{3}]}$ in an $\epsilon$-neighborhood, then the rest of $\wt\gamma$ inside the disk-like open subset $\wt W\cap D_1$.  Therefore, they are in the same connected components in $\Gamma$.  Note that $\wt\gamma_1(s)$ is not necessarily a simple curve.

On the other hand, by definition, $\wt\gamma_1(s)$ is isotopic to $\pi_{J_0}\circ\phi_{H_B}^t\circ\phi_{H_A}^1\circ f(\gamma(s))$ within $\Gamma$.  The fact that no interior point of this family of curves passes through $p_{J_0}$ again follows from the fact that $\phi_{H_B}^t$ preserves $D_i$ for $i>1$.  Taking $t=1$ concludes our lemma.

\end{proof}

By Corollary \ref{c:homotopy_equal}, it follows that $\phi^1_H \circ f|_{D_1}$ is in the trivial mapping class. Therefore, we can find a Hamiltonian isotopy $\phi'$ such that $\phi'\circ\phi^1_H\circ f\in\Stab^0(\Sigma)$. This concludes Theorem \ref{t:trivial_braid} for pure type $\DD$ forms when $n\ge6$. 

\subsection{Conclusion of Theorem \ref{t:trivial_braid} and \ref{thm:main}}\label{sec:mainproof2}

Two cases remain to be addressed: when $n=5$ and when $\w$ is not pure.  We will first consider the case when $n=5$.

\subsubsection{$n=5$}\label{ssec:k=51}

Due to the extra component of $E_1$-curve, we need a new structure for $\CP^2\#5\ov\CP^2$.  From \cite[Lemma 3.24]{LLWsmall} (and \cite{Ev11} in the monotone case), one can find an embedded Lagrangian $\RP^2$ in the complement of a filling divisor $\Sigma$, whose mod-2 homology class is the same as that of $H$. 

We will not repeat the whole proof, but it should be convenient for the reader to recall the sketch of this construction (see also the proof of \cite[Lemma 3.18]{LLWsmall}).  Assuming $\w(H)=1$, by a diffeomorphism induced by Cremona transforms, one may always assume $\w(E_1)<\frac{1}{2}$ and $\sum_{i=1}^5\w(E_i)<2$.  Start with the $SO(3)$-action on $T^*\RP^2$ with a moment map $\mu$.  The symplectic cut along the level set $||\mu||=1$ gives a standard symplectic $\CP^2$ whose line class has symplectic area $1$, and the cut locus is the standard quadric.  Moreover, $||\mu||$ generates the cogeodesic flow in the complement of $\RP^2$, hence induces a circle action on $\CP^2\setminus \RP^2$.  Then one can perform five blow-ups equivariant with respect to this circle action on the standard quadric \cite{Ka99}.  Therefore, the total transform of the standard quadric is a filling divisor $\Sigma_0$.  From the connectedness of the space of filling divisors, one may construct a Hamiltonian diffeomorphism $\psi$ such that $\psi(\Sigma_0)=\Sigma$.  Therefore, the $\psi$-image of the standard $\RP^2$ is what we need.  We denote $\sigma_i$ as the exceptional divisors coming from the equivariant blow-ups for $i=1,\cdots,5$, and $\sigma_0$ as the proper transform of the standard quadric.

Now we perform a further symplectic cut along a level set $||\mu||=\delta>0$ near the Lagrangian $\RP^2$.  Here $\delta$ should be taken small enough so that the level set does not intersect any exceptional curves of class $E_i$.  The resulting symplectic manifold, denoted by $X^*$, is a five-fold blow-up of the fourth Hirzebruch surface, i.e., $X^*:=\PP(\cO\oplus\cO(4))\#5\ov\CP^2$.
It means that $H_2(X^*)$ is generated by 7 classes, $B^*, F^*, E_1^*,\cdots, E_5^*$, where $B^*$ is the class of the zero section of $\cO(4)$ (so $B^* \cdot B^*=4$), $F^*$ is the fiber class and $E_i^*$ are the exceptional classes of the five blow-ups.  

In $X^*$ we now have seven distinguished symplectic curves: $\sigma_i^*$, $i=1,\cdots, 5$, which are the exceptional divisors of the blow-ups of classes $E_i^*$, respectively; $\sigma_0^*$, which is the proper transform of the standard quadric, whose class is $B^*-\sum_{i=1}^5 E_i^*$; and $\sigma_\infty$, which is the cut locus of $||\mu||=\delta$ of class $B^*-4F^*$.  

With the above understood, assume that we have $f\in \Symp_h(X)$ satisfying $d_{C^0}(f,\id)\ll\delta$.  Then $f(\RP^2)\subset ||\mu||^{-1}([0,\delta))$ and $f(\sigma_i)\subset ||\mu||^{-1}((\delta,+\infty))$.  Performing a symplectic cut along $||\mu||=\delta$ as above, we obtain two sets of divisors $\sigma_i^*$ and $f(\sigma_i)^*$ coming from $\sigma_i$ and $f(\sigma_i)$ for $i=0,\cdots, 5$ (note that in general, $f$ does not descend to $X^*$).  Take $J_0^*$ in $\JJ(X^*)^\reg$ such that $\sigma_i^*$ is $J_0^*$-holomorphic for $i=0,\cdots,5,\infty$.  One can also find a compatible almost complex structure, denoted by $J_f^*$, such that  $f(\sigma_i)^*$ are $J_f^*$-holomorphic for $i=0,\cdots,5$  and $J_f^*=J_0^*$ in a neighborhood of $\sigma_\infty^*$. We emphasized that $J_f^*$ is not induced by $f$. It is just a compatible almost complex structure which makes $f(\sigma_i)^*$ pseudo-holomorphic and equals to $J_0^*$ near $\sigma_\infty^*$.

For $J\in\JJ(X^*)^\reg$, we consider the $J$-holomorphic fibration with fiber class $F^*$.  Note that the regularity of $J$ here means both $B^*-\sum_{i=1}^5E_i^*$ and $E_1^*$ are represented by irreducible curves.  Then the fibration has five singular fibers containing a curve in $E_i^*$ class each.  

We may now run the argument in Section \ref{sec:mainproof1} relative to $\sigma_{\infty}$.  For Step A, we define $(\wt U^f)^*$, $\wt W^*$, and $\wt V^*$ as the union of a neighborhood of $\cup_{i=1}^5 D_i^*$ foliated by $F^*$-curves, where $D_i^*$ are $J_0^*$-stable curves with two components of classes $E_i^*$ and $F^*-E_i^*$ for $i=1,\cdots, 5$.  Choose a family of regular $J_t^f$ supported in $(\wt U^f)^*$ connecting $J_f^*$ to $J_0^*$  such that it is independent of $t$ near $\sigma_\infty^*$, this gives a symplectic isotopy from $f(\sigma_i^*)$ to $\sigma_i^*$ for $i=1,\cdots,5$, a symplectic isotopy from $f(\sigma_0^*)$ to the unqiue pseudo-holomorphic curve in the class $B^*$,  and a constant isotopy from $\sigma_\infty^*$ to itself (because the almost complex structure is chosen to be fixed near $\sigma_\infty^*$).
Extend these symplectic isotopies to a Hamiltonian isotopy $\phi^t_{H^*_A}$ supported inside $\wt W^*$ and preserves $\sigma_\infty$.  Step B is also carried over by a family of Hamiltonian diffeomorphisms $\phi^t_{H_B^*}$ supported away from $\sigma_\infty$.  Moreover, arguments in \ref{ssec:trivialMC} shows that $\phi_*:=\phi^1_{H_B^*}\circ \phi^1_{H_A^*}$ induces trivial mapping class in $\sigma_0$.  

Note that although $\phi_*$ is generated by a family of Hamiltonians whose support possibly contains $\sigma_\infty$, all such Hamiltonians preserve $\sigma_\infty$ as a symplectic curve, and the trace of isotopy of $\sigma_i^*$ is disjoint from $\sigma_\infty$ for $i=0,1,\cdots,5$.  Therefore, one may cut off $\phi_*$ so that it is supported away from $\sigma_\infty$ and have the same restriction on $\sigma_i^*$ for $i=0,1,\cdots,5$

Therefore, $\phi_*$ induces a Hamiltonian diffeomorphism $\phi$ in $X$ such that $\phi\circ f|_{\sigma_i}=\id_{\sigma_i}$ for $i=0,\cdots,5$.  This proves $f$ has a trivial symplectic mapping class by Theorem \ref{t:stab0ham}.

\subsubsection{Non-pure cases}

Assume $(X,\w)$ is a non-pure type $\DD$ positive rational surface.
Given a $C^0$ small symplectomorphism $f$ that is close to the identity, we can first isotope the extra exceptional curves back in a neighborhood of the exceptional curves.
 This allows us to descend the symplectomorphism $f$ to a symplectomorphism $f'$ on the blown-down pure type $\DD$ rational surface, and $f'$ is $C^0$ small away from the blow-down symplectic balls.
Since the blow-down symplectic balls are disjoint from $D$, we can isotope $f'(D)$ back to $D$ as in the previous section. It implies that $f'$ is a Hamiltonian diffeomorphism.
By Lemma 4.3 of \cite{LLW22}, it implies that $f$ is a Hamiltonian diffeomorphism as well.

\begin{proof}[Proof of Theorem \ref{thm:main}]
    Let $f_n$ be a sequence of Hamiltonian diffeomorphisms which $C^0$ converges to a symplectomorphism $f_{\infty}$. Then $f_nf_{\infty}^{-1}$ is a sequence of symplectomorphism converging to the identity.
    Therefore, $f_nf_{\infty}^{-1}$ is a Hamiltonian diffeomorphism when $n$ is large, which implies that $f_{\infty}$ is a Hamiltonian diffeomorphism.

    The complement $\Symp(X,\omega)\setminus \Ham(X,\omega)$ can be proved to be closed in the $C^0$-topology in the same way.
\end{proof}

% By \cite[Theorem 1.2.7]{MO15}, there is a path of almost complex structure $\{J_t^*\}_{t\in[0,1]}$ connecting $J_f^*$ and $J_0^*$ which ensures $J_t^*$ representatives of $E_i$ and $2H-E_1-\cdots-E_5$ are embedded, and $\sigma_\infty$ is $J_t^*$-holomorphic  for all $t$.  For homological reasons, this implies the existence of symplectic isotopy 

\section{Further discussion}
\label{sec:further}
\subsection{Floer theoretic approach}
\label{sec:Floer}
\begin{question}
    Can Floer theory help attack Question \ref{quest:C0-closure}?
\end{question}

Given a symplectic manifold $(X,\omega)$, if we know that any non-trivial mapping class of $X$ acts non-trivially on some Floer theoretic invariants (e.g. Hamiltonian/Lagrangian Floer cohomology, Fukaya category) that are $C^0$-rubust, then it should give an affirmative answer to Question \ref{quest:C0-closure}.
%\textcolor{blue}{[This paragraph looks a bit abstract to me.  I assume we need some flavors of relations between $C^0$ geometry and Floer theory in the manner of Jannaud.  I added some comments on closed rational surfaces below.]} {\cym [edited]}

Jannaud has established results which assert that iterations of Lagrangian Dehn twists do not belong to the $C^0$ closure of $\Symp_0(M)$ with mild assumptions using Floer theory \cite{Jannaud}.  One would naturally expect that the same result can be proved for compositions of arbitrary Lagrangian Dehn twists, provided that the resulting composition is nontrivial.  In particular, this will extend Theorem \ref{thm:main} to all positive rational surfaces in light of \cite[Theorem 1.6]{LLW22}.

However, it is an open question whether every non-trivial mapping class of a positive symplectic rational surface acts non-trivially on some Floer theoretic invariants even for type $\DD$, which explains the advantage of our current approach.  However, as an additional illustration of the Floer theoretic approach, we may consider the four dimensional $A_n$- Milnor fibre.  Let 
$$W:=\{(x,y,z)\,|\,x^2+y^2+z^{n+1}=1\}\subset(\CC^3,\omega_{std})$$ 
be endowed with the restricted K\"{a}hler form $\omega$. It is well-known that $W$ is symplectically equivalent to the plumbing of $n$ copies of $T^*S^2$. The zero sections of the plumbed copies are referred to as \textit{standard spheres}. Similarly to the proof of Theorem \ref{thm:main},  to show that $\Ham_c(W,\omega)$ is a connected component, in the $C^0$-topology, of $\Symp_c(W,\omega)$, it suffices to show that if a homotopy class in $\pi_0(\Symp_c(W,\omega))$ has representatives that are arbitrarily $C^0$-close to the identity, then it must be the trivial class. 

Suppose, to reach a contradiction, that $\alpha\in\pi_0(\Symp_c(W,\omega))$ is a non-trivial homotopy class that admits arbitrarily $C^0$-small representatives. It follows from \cite{W14} that $\alpha$ is represented by a composition of Dehn twists along the standard spheres; denote it by $\xi$. Moreover, by \cite[Theorem 1.3, Proposition 3.11]{KS02}, we can find Lagrangian standard spheres $L,L'$ such that 
\begin{equation}
\label{eq:KS_ineq}
HF(L,L')\neq HF(\xi(L),L').
\end{equation}
In fact, by Hamiltonian invariance, this inequality is true for any representative of $\alpha$.

Now, suppose that $\psi$, with $[\psi]=\alpha$, is $C^0$ small enough so that $\psi(L)$ belongs to a Weinstein neighborhood $\mathcal{W}(L)$ of the standard sphere $L$. Observe that $\psi(L)$ is a Lagrangian sphere in $\mathcal{W}(L)\cong D^*S^2$ and, therefore, Hamiltonian isotopic to $L$ by \cite{Hind12}; this in contradiction to \eqref{eq:KS_ineq}.

\subsection{$C^0$-small pseudo-holomorphic isotopy}\label{sec:C0isotopy}

Suppose that $\Sigma$ is a filling divisor of $(M,\omega)$ and $f_n$ are as in Theorem \ref{t:trivial_braid}.
If we were able to find a $C^0$-small Hamiltonian diffeomorphism $\phi_n$ such that $\phi_n \circ f_n(\Sigma)=\Sigma$, then it will be clear that $\phi_n \circ f_n \in Stab^0(\Sigma)$.
The novelty of Theorem \ref{t:trivial_braid} is that even though we don't know that $\phi_n$ is $C^0$-small, we are still able to show that $\phi_n \circ f_n \in Stab^0(\Sigma)$.
But in general, we can ask:

\begin{question}\label{q:C0smallIsotopy}
    Given a symplectic divisor $\Sigma$ and a $C^0$-small symplectomorphism $f$, can one find a $C^0$ small Hamiltonian isotopy $\phi^t$ such that $\phi^1\circ f(\Sigma)=\Sigma$?
\end{question}

Even though we have pseudo-holomorphic isotopy techniques in dimension $4$, the almost complex structures $J$ and $f_*J$ are not $C^0$-close. As a result, it is difficult to find a path of almost complex structures from $J$ to $f_*J$ such that the induced isotopy is controllable.

In the proof of Theorem \ref{t:trivial_braid}, we control the $\pi_{J_0}$ image of the isotopy to obtain our result.

\subsection{Other examples}

Our approach heavily rely on the fact that we know the symplectic mapping class group of the symplectic manifold, for example, the bihopfian property (Lemma \ref{l:bihopfian}) is crucial to the proof of Proposition \ref{p:bihopfian} and hence Theorem \ref{t:stab0ham}. It would be interesting to see examples that we can answer Question \ref{quest:C0-closure} while we don't know the symplectic mapping class group.
Type $\EE$ symplectic rational surfaces could be the next class of examples to test our understanding.

% \printbibliography
\bibliographystyle{amsalpha}
\bibliography{biblio}

\end{document}